\documentclass[12pt,a4paper]{article}
\usepackage[centertags]{amsmath}
\usepackage{amsfonts}
\usepackage{amssymb}
\usepackage{amsthm}
\usepackage{newlfont}
\usepackage[T1]{fontenc}
\usepackage[ansinew]{inputenc}
\usepackage{amsmath}
\usepackage{graphicx}
\usepackage{color}
\usepackage[ansinew]{inputenc}
\usepackage[all]{xy}
\usepackage{fancyhdr}
\theoremstyle{plain}
\newtheorem{thm}{Theorem}[subsection]
\newtheorem{cor}[thm]{Corollary}
\newtheorem{lem}[thm]{Lemma}
\newtheorem{prop}[thm]{Proposition}

\newtheorem{newthm}{Theorem}[subsection]
\theoremstyle{definition}
\newtheorem{defn}[thm]{Definition}

\newtheorem{rem}[thm]{Remark}

\newcommand{\al}{\alpha}
\newcommand{\be}{\beta}
\newcommand{\f}{\varphi}
\newcommand{\de}{\delta}
\newcommand{\s}{\sigma}

\newcommand{\dd}{\partial}
\newcommand{\g}{\gamma}

\newcommand{\G}{\Gamma}
\newcommand{\Si}{\Sigma}

\newcommand{\D}{\Delta}

\renewcommand{\k}{\kappa}
\renewcommand{\t}{\tau}
\renewcommand{\L}{\mathcal{B}}
\newcommand{\I}{\mathcal{I}}

\newcommand{\N}{\mathbb N}

\newcommand{\Z}{\mathbb Z}
\newcommand{\Q}{\mathbb Q}
\newcommand{\C}{\mathbb C}
\newcommand{\Or}{\mathcal O}

\newcommand{\Ø}{\emptyset}

\newcommand{\K}{\mathcal{K}}

\newcommand{\F}{\mathcal{F}}
\newcommand{\rank}{\textup{rank}}

\newcommand{\im}{\textup{Im}}
\newcommand{\rk}{\textup{rk}}
\newcommand{\Sp}{\textup{Sp}}

\newcommand{\link}{\textup{link}}
\renewcommand{\star}{\textup{star}}
\newcommand{\iso}{\cong}

\newcommand{\set}[1]{\left\{#1\right\}}

\newcommand{\To}{\longrightarrow}
\newcommand{\Tto}{\Rightarrow}
\newcommand{\into}{\hookrightarrow}

\newcommand{\abs}[1]{\left\vert#1\right\vert}

\newcommand{\id}{\textup{id}}
\newcommand{\pr}{\textup{pr}}
\newcommand{\ideal}[1]{\langle #1 \rangle}
\newcommand{\idealQ}[1]{\langle #1 \rangle_{\Q}}

\newcommand{\del}{\subseteq}
\newcommand{\fra}{\setminus}

\newcommand{\tensor}{\otimes}

\newcommand{\xort}{\ideal{x}^{\perp}}
\newcommand{\xyort}{\ideal{x,y}^{\perp}}
\newcommand{\ialg}[1]{i_{\textup{alg}}\!\left(#1\right)}
\newcommand{\Lconn}[1]{\L^{#1}_{Z,x,k}}
\newcommand{\Lconnk}[1]{\L^{2}_{Z,x,#1}}
\newcommand{\Lconnxk}[2]{\L^{1}_{Z,#1,#2}}
\newcommand{\LconnZ}[1]{\L^{1}_{#1,x,k}}
\newcommand{\Letdel}{\L^{\D^k}}
\newcommand{\Letdelx}{\L^{\D^k;\xort}}
\newcommand{\Letblank}[1]{\L^{#1}}
\newcommand{\Letblankx}[1]{\L^{#1;\xort}}
\newcommand{\Ll}{\L^{a_{1}}}
\newcommand{\Lblank}[1]{\L^{a_{1},#1}}
\newcommand{\Ldel}{\L^{a_{1},\D^k}}
\newcommand{\Ldelpos}{\L^{a_{1}}_{\gcd\ne0}}
\newcommand{\Ldelgode}{\L^{a_{1},\D^k}_{\gcd=1}}

\newcommand{\Lblankgode}[1]{\L^{a_{1},#1}_{\gcd=1}}
\newcommand{\Ldelgodeb}{\L^{a_{1},\D^k}_{\gcd=1;\,t}}
\newcommand{\Lblankgodeb}[1]{\L^{a_{1},#1}_{\gcd=1;\,t}}
\newcommand{\Mdel}{\mathcal{M}^{\D_1,\D_2,\D_3}_{D(\D_1)\,\mid\, D(\D_2)}}
\newcommand{\Mdelblank}[3]{\mathcal{M}^{#3}_{#1\,\mid\, #2}}

\newcommand{\Ndel}{\mathcal{N}_{\D}}
\newcommand{\Ndelblank}[1]{\mathcal{N}^{#1}_{\D}}
\newcommand{\Ndelblankx}[1]{\mathcal{N}^{#1; \xort}_{\D}}
\newcommand{\vv}{\mathfrak{v}}
\newcommand{\ww}{\mathfrak{w}}
\newcommand{\uu}{\mathfrak{u}}
\newcommand{\gcdto}{\textup{gcd}_2}
\newcommand{\lediv}{\le_{\textup{div}}}
\newcommand{\gediv}{\ge_{\textup{div}}}

\newcommand{\multi}{multi-simplicial complex}

\newcommand{\multis}{multi-simplicial complexes}
\newcommand{\Multis}{Multi-simplicial complexes}
\newcommand{\simp}{\textup{simp}}
\newcommand{\Ltreort}[1]{\Lambda^3\idealQ{#1}^{\perp}}
\newcommand{\LtreH}{\Lambda^3H_\Q}
\newcommand{\Ltre}{\Lambda^3}
\newcommand{\Lort}[2]{\Lambda^{#1}\idealQ{#2}^{\perp}}
\newcommand{\ortZ}[1]{\ideal{#1}^\perp}
\newcommand{\ortQ}[1]{\idealQ{#1}^\perp}

\newcommand{\streg}[1]{\overline{#1}}
\newcommand{\bi}{b^i}
\catcode`\=13
\def{$\bowtie$}
\author{Søren K. Boldsen\footnote{This author is currently supported by a grant from the Carlsberg Foundation.}{ } and Mia Hauge Dollerup}
\title{Towards representation stability for the second homology of the Torelli group}
\date{\today}

\begin{document}
\begin{titlepage}\maketitle\thispagestyle{empty}
\begin{abstract}
We show for $g\ge 7$ that the second homology group of the Torelli group, $H_2(\I_{g,1};\Q)$, is generated as an $\Sp(2g,\Z)$-module by the image of $H_2(\I_{6,1};\Q)$ under the stabilization map. In the process we also show that the quotient $B(F_{g,i};i)/\I_{g,i}$ of the complex of arcs with identity permutation by the Torelli group  is $(g-2)$-connected, for $i=1,2$. 
\end{abstract}
\tableofcontents\end{titlepage}\newpage
\pagestyle{fancy}
\section{Introduction}
Let $F_{g,r}$ denote a smooth compact connected oriented surface of genus $g$ and $r$ boundary components. Let $\G_{g,r}=\G(F_{g,r})$ denote its mapping class group, i.e. $\G(F)=\pi_0(\textup{Diff}^+(F,\dd F))$, where $\textup{Diff}^+(F,\dd F)$ is the group of orientation-preserving diffeomorphisms of $F$ that restrict to the identity on $\dd F$. 

The Torelli group $\I_{g,1}$ is the subgroup of $\G_{g,1}$ defined by the exact sequence
\begin{equation}\label{e:Torelli}
	1\To \I_{g,1}\To \G_{g,1}\To \Sp(2g,\Z)\To 1.
\end{equation}
To define the Torelli group of a surface with more than one boundary component, we proceed as in \cite{Putman1}. Suppose we have an embedding $S\To F_{g,1}$ such that $F_{g,1}\fra S$ is connected. Write $\G(F_{g,1},S)$ for the image of $\G(S)$ in $\G_{g,1}$ under the map induced by this embedding. Then one defines,
\begin{equation}
	\I(F_{g,1},S):= \I(F_{g,1})\cap \G(F_{g,1},S).
\end{equation}
In this paper we are interested in the case $S=F_{g-1,2}$, and the embedding is $\Si_{1,-1}:F_{g-1,2}\To F_{g,1}$ which glues on a pair of pants. We write  $\I_{g-1,2}$ for $\I(F_{g,1},F_{g-1,2})$ defined via this embedding. 

There is an exact sequence similar to \eqref{e:Torelli}, as follows: Let $\be$ an arc such that $F_{g-1,2}\To F_{g,1}$ is the inclusion of the cut-up surface $(F_{g,1})_\be\To F_{g,1}$ as on Figure \ref{f:beta}, denote by $\tilde \be$ the closing-up of $\be$ (see Figure \ref{f:closeup}), and let $b=[\tilde\be]\in H_1(F_{g,1};\Z)$ be its homology class. Then
\begin{equation}\label{e:Torelli2}
	1\To \I_{g-1,2}\To \G(F_{g,1},F_{g-1,2})\To \Sp(2g,\Z)_b\To 1.
\end{equation}
where $\Sp(2g,\Z)_b\del \Sp(2g,\Z)$ is the stabilizer subgroup for $b$. 

\begin{figure}[!hbt]
\begin{center}
\setlength{\unitlength}{0.5cm}
\begin{picture}(7,1.8)(0,.7)
\color[rgb]{0.7,0.7,0.7}\qbezier(4,0.5)(4.5,0)(4,-0.5)
\qbezier(4,1.5)(4.5,2)(4,2.5)\color{black}
%xi
\qbezier(6.8,2)(5.5,2)(5,1)\color[rgb]{0.7,0.7,0.7}\qbezier(5,1)(6,1)(6.3,-.5)\color{black}
\qbezier(6.3,-.5)(6.5,0)(6.8,0)
\put(1.3,0.7){$F_{g,2}$}\put(7.6,0.7){$F_{g+1,1}$}
\qbezier(4,-0.5)(1,-0.5)(1,-0.5)\qbezier(4,2.5)(1,2.5)(1,2.5)
\qbezier(4,0.5)(3.5,0)(4,-0.5) \qbezier(4,1.5)(3.5,2)(4,2.5)
\qbezier(4,0.5)(2,1)(4,1.5) \qbezier(4,0.5)(6,1)(4,1.5)
\qbezier(7,-0.5)(6.3,1)(7,2.5) \qbezier(7,-0.5)(5.5,-0.5)(4,-0.5)
\qbezier(7,2.5)(7,2.5)(4,2.5)
\qbezier(7,-0.5)(7.7,1)(7,2.5)
%\qbezier(3.775,0)(6.5,0)(6.8,0)
%\put(3.775,0){\circle*{0.2}}
\put(6.8,0){\circle*{0.2}}
%\qbezier(3.775,2)(6.5,2)(6.8,2)
%\put(3.775,2){\circle*{0.2}}
\put(6.8,2){\circle*{0.2}}
\put(5.8,1.1){$\be$}
\put(3,-0.3){$\tilde \be$}
\end{picture}
\end{center}\caption{The arc $\be$ such that $F_{g-1,2}=(F_{g,1})_\be$, and its close-up $\tilde \be$.}\label{f:beta}\end{figure}
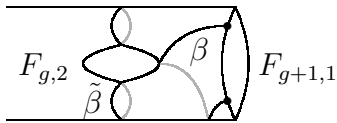
We can now state our main theorem, which is part of Conjecture 6.1 of \cite{Repstab} (more below):
\begin{thm}\label{t:main}
Let $g\ge 7$. The image of the map induced by $F_{g-1,1}\to F_{g,1}$,
\begin{equation*}
    H_2(\I_{g-1,1};\Q)\To H_2(\I_{g,1};\Q),
\end{equation*}
generates $H_2(\I_{g,1};\Q)$ as an $\Sp(2g;\Z)$-module.

As a consequence, $H_2(\I_{g,1};\Q)$ is generated as an $\Sp(2g,\Z)$-module by the image of $H_2(\I_{6,1};\Q)$.
\end{thm}

We will investigate the group homology of the Torelli group $\I_{g,i}$ ($i=1,2$) via a spectral sequence for the action of $\I_{g,i}$ on a highly connected complex $B_*(F_{g,i};i)$: Given a $d$-connected complex $X=\set{X_p}_{p\ge 0}$ with a rotation-free action of a group $G$, there is an augmented spectral sequence $E^*_{p,q}(X)$ with 
\begin{equation}
	E^1_{p,q}(X) \iso \bigoplus_{\s\in \D_{p-1}} H_q(G_\s) \quad \Tto 0, \quad \text{for } p+q\le d+1,
\end{equation}
where $G_{\s}\del G$ is the stabilizer subgroup for the simplex $\s$, and $\D_p$ denotes a set of representatives for the orbit set $X_p/G$. See e.g. \cite{Brown}, VII \S 7.

We now define $B_*(F;i)$. First, recall Harer's arc complex $C_*(F;i)$, where $i\in\set{1,2}$; see \cite{Harer}. This is the simplicial complex whose $n$-simplices are $n+1$ isotopy classes of arcs joining two fixed points on $\dd F$ (if $i=1$, the points are on the same boundary component, if $i=2$ they are on different boundary components); the arcs must be disjoint (away from endpoints) and must not disconnect $F$. $C_*(F;i)$ has an obvious rotation-free action of $\G(F)$.

A simplex gives rise to a permutation; namely, given an order of the arcs at the starting point, how the arcs are permuted at the ending point (read off with the opposite orientation, by convention). Then $B_*(F;i)$ is the subcomplex of $C_*(F;i)$ of simplices with the identity permutation. For an illustration, see the left part of Figure \ref{f:closeup}. It has been shown that $B_*(F_{g,r};i)$ is $(g-3+i)$-connected by \cite{Ivanov1} Thm. 3.5 for $i=1$, and in the general case by \cite{RW}, Thm. A1.

For $F_{g,1}$, the action of $\G_{g,1}$ restricts to an action of $\I_{g,1}$. For $F_{g-1,2}$, we embed $B_*(F_{g-1,2};2)$ into $B_*(F_{g,1};1)$ by extending the arcs of each arc simplex parallelly along two fixed disjoint arcs in the pair of pants, as shown in Figure \ref{f:extendarcs}. Consequently, $\I_{g-1,2}$ acts on $B_*(F_{g-1,2};2)$.
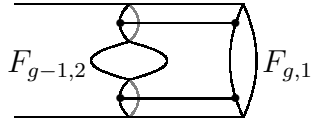
\begin{figure}[!htb]
\begin{center}
\setlength{\unitlength}{0.5cm}
\begin{picture}(7,1.6)(0,0.7)
\color[rgb]{0.5,0.5,0.5}\qbezier(4,0.5)(4.5,0)(4,-0.5)
\qbezier(4,1.5)(4.5,2)(4,2.5)\color{black}
\put(.8,0.7){$F_{g-1,2}$}\put(7.5,.7){$F_{g,1}$}
\qbezier(4,-0.5)(1,-0.5)(1,-0.5)\qbezier(4,2.5)(1,2.5)(1,2.5)
\qbezier(4,0.5)(3.5,0)(4,-0.5) \qbezier(4,1.5)(3.5,2)(4,2.5)
\qbezier(4,0.5)(2,1)(4,1.5) \qbezier(4,0.5)(6,1)(4,1.5)
\qbezier(7,-0.5)(6.3,1)(7,2.5) \qbezier(7,-0.5)(5.5,-0.5)(4,-0.5)
\qbezier(7,2.5)(7,2.5)(4,2.5) \qbezier(7,-0.5)(7.7,1)(7,2.5)
\qbezier(3.775,0)(6.5,0)(6.8,0)\put(3.775,0){\circle*{0.2}}\put(6.8,0){\circle*{0.2}}
\qbezier(3.775,2)(6.5,2)(6.8,2)\put(3.775,2){\circle*{0.2}}\put(6.8,2){\circle*{0.2}}
\end{picture}
\end{center}\caption{Extending arc simplices $B_*(F_{g-1,2};2)\To B_*(F_{g,1};1)$}\label{f:extendarcs}\end{figure}

The main result of the first part, which is proved in Section \ref{ss:conn}, is then
\begin{thm}\label{t:intro}
The quotient complex $B_*(F_{g,i};i)/\I_{g,i}$ is $(g-2)$-connected, for $i=1,2$.
\end{thm}

We can now outline the proof of our main theorem.

\begin{proof}[Proof of Theorem \ref{t:main}.]
The stabilization map $F_{g-1,1}\To F_{g,1}$ is a composition $\Si_{1,-1}\circ \Si_{0,1}: F_{g-1,1}\To F_{g-1,2}\To F_{g,1}$, where the $\Si_{i,j}$ are the maps that glue on a pair of pants.

For $\Si_{1,-1}:F_{g-1,2}\To F_{g,1}$ we use the spectral sequence $E^r_{p,q}(F_{g,1};1)$ for the action of $\I_{g,1}$ on the arc complex $B_*(F_{g,1};1)$. Since $B_*(F_{g,1};1)$ is $(g-2)$-connected, we obtain
\begin{equation}\label{e:spectral}
    E^1_{p,q}(F_{g,1};1)\iso \bigoplus_{\al\in\D_{p-1}}H_q(\I(F_{g,1})_{\al};\Q)\Tto 0\text{ for } p+q\le g-1.
\end{equation}
Here, $\D_{p}$ denotes a set of representatives for the orbits $B_p(F;1)/\I(F)$, and $\I(F)_{\al}$ is the stabilizer subgroup of $\al$ in $\I(F)$. Just as for mapping class groups, one shows that $\I(F)_{\al}\iso \I(F)\cap\G(F,F_{\al})=\I(F,F_\al)$, where $F_\al$ denotes $F$ cut up along $\al$. Also, by Prop. \ref{p:complex} $(i)$, $B_p(F_{g,1};1)/\I_{g,1}\iso \L(g)^{(p)}$. 

\noindent Choosing a section $T:\L(g)\To B_*(F_{g,1};1)$, we can rewrite \eqref{e:spectral} as
\begin{equation}\label{e:spectral2}
    E^1_{p,q}(F_{g,1};1)\iso \hspace{-0.3cm} \bigoplus_{\ww\in\L(g)^{(p-1)}}\hspace{-0.3cm} H_q(\I(F_{g,1},(F_{g,1})_{T(\ww)});\Q)\Tto 0\text{ for } p+q\le g-1.
\end{equation}
In particular, for $\al$ a $0$-simplex, we have $(F_{g,1})_{\al}\iso F_{g-1,2}$, and each component map of the differential
\begin{equation*}
d^1_{1,2}: \bigoplus_{w\in \L(g)^{(0)}}H_2(\I(F_{g,1};(F_{g,1})_{T(w)});\Q)\To  H_2(\I(F_{g,1});\Q)
\end{equation*}
is precisely the map $d^1_{1,2}(w): H_2(\I(F_{g,1};F_{g-1,2});\Q)\To  H_2(\I(F_{g,1});\Q)$ induced by $\Si_{1,-1}$.  The differential is $\Sp(2g,\Z)$-equivariant, which can be seen from the construction of $E^*_{p,q}(F_{g,1};1)$, using the resolution of $\G_{g,1}$ instead of $\I_{g,1}$, and applying \eqref{e:Torelli}. So since $\Sp(2g,\Z)$ acts transitively on the $0$-simplices $\L(g)^{(0)}$, the image of $d^1_{2,1}$ equals the $\Sp(2g,\Z)$-module generated by the image of just one component map $d^1_{2,1}(w)=(\Si_{1,-1})_*$. We will show $d^1_{2,1}$ is surjective.

For the map $\Si_{0,1}: F_{g-1,1}\To F_{g-1,2}$, we use the spectral sequence $E^r_{p,q}(F_{g-1,2};2)$ for the action of the Torelli group $\I_{g-1,2}=\I(F_{g,1};F_{g-1,2})$ on the arc complex $B_*(F_{g-1,2};2)$. We have $B_p(F_{g-1,2};2)/\I_{g-1,2}\iso \Ll(g-1)$ by Prop. \ref{p:complex} $(ii)$, so in a similar manner as above we obtain
\begin{equation}\label{e:spectral3}
    E^1_{p,q}(F_{g-1,2};2)\iso \hspace{-0.5cm}\bigoplus_{\ww\in\L^{a_1}(g)^{(p-1)}}\hspace{-0.5cm}H_q(\I(F_{g,1},(F_{g-1,2})_{T(\ww)});\Q) \Tto 0\text{ for } p+q\le g-1.
\end{equation}
For a $0$-simplex $\al$, we have $(F_{g-1,2})_{\al}\iso F_{g-1,1}$ and 
$d^1_{2,1}$ has component maps equal to the map induced by $\Si_{0,1}$. The stabilizer subgroup $\Sp(2g,\Z)_{b}$ acts transitively on $\Ll(g-1)^{(0)}$, and $d^1_{2,1}$ is $\Sp(2g,\Z)_b$-equivariant from \eqref{e:Torelli2}. Thus the image of $d^1_{2,1}$ equals the $\Sp(2g,\Z)_{b}$-module generated by the image of a single component map $(\Si_{0,1})_*$. We will show $d^1_{2,1}$ is surjective.

We see that to prove the main theorem, we must show that the differential $d^1_{1,2}$ is surjective in both spectral sequences. To do this, since the spectral sequence converges to zero, it suffices to show that $E^2_{3,0}=0$ and $E^2_{2,1}=0$.

To show $E^2_{3,0}=0$, note $E^1_{p,0}(F_{g,1};1)=\bigoplus_{\ww\in\L(g)^{(p-1)}}\Q=C_{p-1}(\L(g);\Q)$ is the $(p-1)$st chain group of the augmented chain complex for $\L(g)$ with $\Q$-coefficients. Similarly, $E^1_{p,0}(F_{g-1,2};2)=C_{p-1}(\Ll(g-1);\Q)$. Since $\L(g)$ is $(g-2)$-connected, and $\Ll(g-1)$ is $(g-3)$-connected by Theorem \ref{t:intro}, the homology of the chain complexes is zero in degrees $\le g-3$. So since $g\ge 7$, we see $E^2_{3,0}=0$ in both cases.

This reduces the proof of the Theorem to showing that $E^2_{2,1}=0$ in both spectral sequences. This will be done in section \ref{ss:exact}.
\end{proof}
We briefly mention some accessible improvements of our results: Firstly, extending the result to any number of boundary components, and showing stability for $\Z$-coefficients instead of $\Q$-coefficients; both should be possible by the results of \cite{Berg}. Secondly, stability for lower genus, but using \cite{Berg} would either require improving her results, or only reduce the genus by 1, since in Theorem \ref{t:Berg}, the genus of $S$ must be at least $3$.

\vspace{0.3cm}
We close this introduction by placing the result of the Main Theorem into a larger context: The motivation behind this paper is the question of whether the second Morita-Miller-Mumford class $\k_2\in H^4(\G_{g,1})$ restricts non-trivially to $H^4(\I_{g,1})$ or not. One can approach this question by attempting to use the spectral sequence for the fibration \eqref{e:Torelli}, for which we must investigate $H^p(\Sp(2g;\Z);H^q(\I_{g,1}))$, we focus here on $q=2$. Such groups have been studied stably by \cite{Borel}, see in particular Theorem 4.4. One way to to show that the requirements of the theory are fulfilled would be to show that the $\Sp(2g,\Z)$-representations behave well under the stabilization map $H_2(\I_{g,1})\To H_2(\I_{g+1,1})$. The formalization of this is what \cite{Repstab} has termed representation stability, and they conjectured this for $H_q(\I_{g,1};\Q)$.

More precisely, the Main Theorem is basically one of the four conditions (namely, surjectivity) for representation stability of $H_2(\I_{g,1};\Q)$; the others are injectivity, rationality, and stability of the multiplicities of the irreducible representations. We have no results for these three conditions, though injectivity might be solved in a similar manner.
Another direction would be to show the Main Theorem for higher homology degrees. Our proof that $E^2_{2,1}=0$ is computational and specific to $H_1(\I_{g,1})$, and so is not readily generalizable. For the question of $\k_2$, representation stability for $H_3(\I_{g,1})$ would also be needed. 

\section[Connectivity of the quotient of the arc complex]{Connectivity of the quotient of the arc complex by the Torelli group}\label{ss:conn}

\subsection{A concrete description of the quotient complexes}
In this section, $H=H_1(F_{g,1};\Z)$. Let $\set{\al_1,\beta_1,\ldots,\al_g,\be_g}$ be a standard set of simple closed curves on $F_{g,1}$ as on Figure \ref{f:stdcurves}, with homology classes $a_i=[\al_i]$ and $b_i=[\be_i]$, such that $\set{a_1,b_1,\ldots,a_g,b_g}$ is a symplectic basis for $H$ with respect to the intersection form $\ialg{-,-}$. Make the convention that the curve $\be_1$ is one of the boundary components of $F_{g-1,2}$. 

\begin{figure}[!h]
\begin{center}
\setlength{\unitlength}{0.4cm}
\begin{picture}(14,7.5)(0,-3)
%betakurver
\qbezier(3,.35)(3.8,1.7)(3,3)
\color[rgb]{0.8,0.8,0.8}\qbezier(3,.35)(2.2,1.7)(3,3)\color{black}
\qbezier(7,.35)(7.8,1.7)(7,3)
\color[rgb]{0.8,0.8,0.8}\qbezier(7,.35)(6.2,1.7)(7,3)\color{black}
\qbezier(12,.35)(12.8,1.7)(12,3)
\color[rgb]{0.8,0.8,0.8}\qbezier(12,.35)(11.2,1.7)(12,3)\color{black}
%flade og navne
\qbezier(0,3)(-1,0)(0,-3)\qbezier(0,3)(1,0)(0,-3)
\put(0,3){\line(1,0){13}}\put(0,-3){\line(1,0){13}}
\put(9.1,0){$\ldots$}\put(0.8,-1){$\al_1$}\put(6.8,-2.2){$\al_2$}\put(11.8,-2.2){$\al_g$}
\put(3.5,2){$\be_1$}\put(7.5,2){$\be_2$}\put(12.5,2){$\be_g$}
\put(-2.4,0){$F_{g,1}$}
\put(3,3.2){$\overbrace{\hspace{4.4cm}}^{F_{g-1,2}}$}
%alfakurver
\qbezier(1.5,0)(1.6,1.4)(3,1.5)\qbezier(4.5,0)(4.4,1.4)(3,1.5)
\qbezier(1.5,0)(1.6,-1.4)(3,-1.5)\qbezier(4.5,0)(4.4,-1.4)(3,-1.5)
\qbezier(5.5,0)(5.6,1.4)(7,1.5)\qbezier(8.5,0)(8.4,1.4)(7,1.5)
\qbezier(5.5,0)(5.6,-1.4)(7,-1.5)\qbezier(8.5,0)(8.4,-1.4)(7,-1.5)
\qbezier(10.5,0)(10.6,1.4)(12,1.5)\qbezier(13.5,0)(13.4,1.4)(12,1.5)
\qbezier(10.5,0)(10.6,-1.4)(12,-1.5)\qbezier(13.5,0)(13.4,-1.4)(12,-1.5)
%F_{g-1,2} rand
\color[rgb]{0.8,0.8,0.8}\qbezier(3,-.22)(2.2,-1.7)(3,-3)\color{black}
\qbezier(3,-.22)(3.8,-1.7)(3,-3)
%genushuller
\qbezier(2,.17)(3,-.63)(4,.17)\qbezier(2.2,0)(3,.67)(3.8,0)
\qbezier(6,.17)(7,-.63)(8,.17)\qbezier(6.2,0)(7,.67)(7.8,0)
\qbezier(11,.17)(12,-.63)(13,.17)\qbezier(11.2,0)(12,.67)(12.8,0)
\qbezier(13,3)(14.5,3)(14.5,0)\qbezier(14.5,0)(14.5,-3)(13,-3)
\end{picture}
\caption{Simple closed curves giving a symplectic basis of $H=H_1(F_{g,1})$.}\label{f:stdcurves}\end{center}\end{figure}
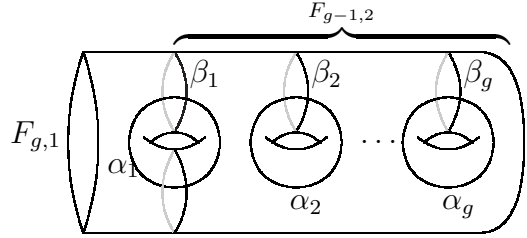
\vspace{-0.7cm}
\begin{defn}\label{d:rank} Given a symplectic basis $\set{a_1,b_1,\ldots,a_g,b_g}$ of $H$. For $x\in H$, express $x$ in the basis as $x=\sum_{i=1}^g(c_ia_i+d_ib_i)$. Then the $a_j$-rank of $x$ is  $\rk^{a_j}(x):=c_j$. Similarly, $\rk^{b_j}(x):=d_j$.
\end{defn}
Note: We have $\rk^{a_j}(x)=\ialg{x,b_j}$ and $\rk^{b_j}(x)=-\ialg{x,a_j}$.

\begin{defn}$\L(g)$ is the complex where each $n$-simplex is an ordered basis $(x_0,x_1,\ldots,x_n)$ for an isotropic summand of $H$. 

Let $\Ll(g-1)$ denote the subcomplex of $\L(g)$ given by those ordered isotropic bases $(x_1,\ldots, x_n)$ with $\rk^{a_{1}}(x_i)=1$ for all $i=1,\ldots,n$.
\end{defn}
Note, $\L(g)$ and $\Ll(g-1)$ are not simplicial complexes, but a simplex is determined by its vertices and their ordering. More in section \ref{ss:prelim}.

\begin{prop}[\cite{Berg}]\label{p:complex}\quad
\begin{itemize}
	\item[$(i)$] The quotient complex $B_*(F_{g,1};1)/\I_{g,1}$ is isomorphic to $\L(g)$.
	\item[$(ii)$] The quotient complex $B_*(F_{g-1,2};2)/\I_{g-1,2}$ is isomorphic to $\Ll(g-1)$.
\end{itemize}
\end{prop}

\begin{proof}Our original inspiration for this result and its proof was  \cite{Putman2} Lemma 6.9, where he showed that the quotient of the complex of simple closed curves in $F_{g,r}$ by the Torelli group, $D_*(F_{g,r})/\I_{g,r}$, is isomorphic to the simplicial complex $\mathcal{L}(g)$ of lax isotropic bases of $H_1(F_{g,r})$. But we have later discovered that \cite{Berg} in her ph.d. thesis in 2003, Prop. 2.5.3, has proven the result for $C_*(F_{g,r};i)/\I_{g,r}$. We can deduce the Proposition by restricting to $B_*(F_{g+1-i,i};i)$, i.e. requiring that the permutation $=\id$.
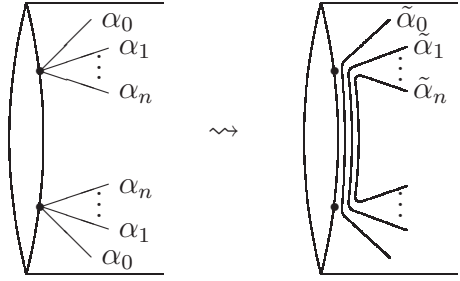
\begin{figure}[!h]
\begin{center}
\setlength{\unitlength}{0.45cm}
\begin{picture}(4,8)(0,-4)
\qbezier(0,4)(-1,0)(0,-4)\qbezier(0,4)(1,0)(0,-4)
\put(0,4){\line(1,0){4}}\put(0,-4){\line(1,0){4}}
\put(0.4,2){\circle*{0.2}}\put(0.4,-2){\circle*{0.2}}
\put(0.4,2){\line(1,1){1.5}}\put(0.4,2){\line(3,-1){2}}\put(0.4,2){\line(3,1){2}}
\put(0.4,-2){\line(1,-1){1.5}}\put(0.4,-2){\line(3,1){2}}\put(0.4,-2){\line(3,-1){2}}
\put(2,1.75){$\vdots$}\put(2,-2.3){$\vdots$}
\put(2.2,3.3){$\al_0$}\put(2.7,2.5){$\al_1$}\put(2.7,1.2){$\al_n$}
\put(2.2,-3.7){$\al_0$}\put(2.7,-2.9){$\al_1$}\put(2.7,-1.6){$\al_n$}
\end{picture}
\begin{picture}(4,8)(0,-4)
\put(1,0){$\rightsquigarrow$}
\end{picture}
\begin{picture}(4,8)(0,-4)
\qbezier(0,4)(-1,0)(0,-4)\qbezier(0,4)(1,0)(0,-4)
\put(0,4){\line(1,0){4}}\put(0,-4){\line(1,0){4}}
\put(0.4,2){\circle*{0.2}}\put(0.4,-2){\circle*{0.2}}
\qbezier(0.6,2.1)(0.63,2.26)(0.7,2.3)
\qbezier(0.7,2.3)(2,3.5)(2,3.5)
\qbezier(0.6,2.1)(0.8,0)(0.6,-2.1)
\qbezier(0.6,-2.1)(0.63,-2.26)(0.7,-2.3)
\qbezier(0.7,-2.3)(2,-3.5)(2,-3.5)
\qbezier(0.9,2.1)(2.5,2.7)(2.5,2.7)
\qbezier(0.9,2.1)(0.82,2.06)(0.8,1.9)
\qbezier(0.8,1.9)(1,0)(0.8,-1.9)
\qbezier(0.9,-2.1)(2.5,-2.7)(2.5,-2.7)
\qbezier(0.9,-2.1)(0.82,-2.06)(0.8,-1.9)
\qbezier(1.2,1.85)(2.5,1.4)(2.5,1.4)
\qbezier(1.2,1.85)(1,1.9)(1,1.7)
\qbezier(1,1.7)(1.2,0)(1,-1.7)
\qbezier(1.2,-1.85)(2.5,-1.4)(2.5,-1.4)
\qbezier(1.2,-1.85)(1,-1.9)(1,-1.7)
\put(2.2,1.7){$\vdots$}\put(2.2,-2.3){$\vdots$}
\put(2.2,3.3){$\tilde\al_0$}\put(2.7,2.5){$\tilde\al_1$}\put(2.7,1.2){$\tilde\al_n$}
\end{picture}
\caption{Closing up the arcs of a simplex in $B_*(F_{g,1};1)$.}\label{f:closeup}\end{center}\end{figure}

We briefly mention the map that gives the isomorphism. First consider the map $h:B_n(F_{g,1};1)\To \L(g)$ by
\begin{equation*} (\al_0,\al_1,\ldots,\al_n) \mapsto ([\tilde\al_0],[\tilde\al_1],\ldots,[\tilde\al_n])
\end{equation*}
where the simple closed curve $\tilde\al_i$ comes from closing up $\al_i$ as on Figure \ref{f:closeup}, and $[-]$ denotes the homology class. The closing-up is always possible, and gives non-intersecting curves, because the arc simplices have permutation $\id$. 
When we restrict to the subcomplex $B_*(F_{g-1,1};2)$, the target of $h$ is indeed contained in $\Ll(g-1)$; for the close-up $\tilde\g$ of an arc $\g$ in $B_{*}(F_{g,1};1)$ coming from $B_{*}(F_{g,2};2)$ will satisfy $i_{\textup{geom}}(\tilde\g,\be_1)=1$.

Since $\I_{g,1}$ preserves homology classes, $h$ descends to a map on the quotient $\bar h:B_*(F_{g,1};1)/\I_{g,1}\To\L(g)$. Then $\bar h$ is a bijection.
\end{proof}

\subsection{\Multis{} and preliminaries}\label{ss:prelim}
The complexes $\L(g)$ and $\Ll(g)$, along with the other complexes we will define here in section 2, are of the following type, which we call \multis{}, for lack of a better word:

\begin{defn}
A nonempty family $\K$ of finite ordered tuples of a universal set $H$ is called a \multi{} if, 
\begin{itemize}
	\item[$(i)$] for every tuple $\ww\in\K$, and every sub-tuple $\vv$ of $\ww$, we have $\vv\in \K$. (By a sub-tuple of $\ww=(w_0,\ldots,w_n)$ we mean a tuple $(w_{i_0},\ldots,w_{i_j})$ where $0\le i_0<\cdots<i_j\le n$.)
	\item[$(ii)$] whether $\ww\in K$ does not depend on the ordering of the tuple $\ww$.
\end{itemize}
An $(n+1)$-tuple $\ww$ in $\K$ will be called an $n$-simplex.
\end{defn}
\begin{rem}\label{r:join}
	This allows the following combinatorial definitions:
	Let $\vv=(v_0,\ldots,v_n)$ and $\ww=(w_0,\ldots,w_k)$ be simplices in $\K$. The vertex-set of $\ww$ is $V(\ww)=\set{w_0,\ldots,w_n}$. 
	The link of $\vv$, $\link_{\K}(\vv)$, is defined to be the set of all simplices $\ww\in\K$ such that $(v_0,\ldots,v_n,w_0,\ldots,w_k)$ is a simplex in $\K$.
	For $\ww\in\link_{\K}(\vv)$, we say a simplex $\uu$ is a join of $\vv$ and $\ww$, if $V(\uu)=\set{v_0,\ldots,v_n, w_0,\ldots,w_k}$. We will write, by slight abuse of notation, $\uu=\vv*\ww$ to mean that $\uu$ is a join of $\vv$ and $\ww$. 
	Note, $\link_{\K}(\vv*\ww)$ is unambiguous.
\end{rem}
 
To say a \multi{} $\K$ is $d$-connected means that the geometric realization, $\abs{\K}$, is $d$-connected. The standard proof for simplicial approximation works equally well to show that a map $f:\abs{S}\To \abs{\K}$, where $S$ is a simplicial complex, is homotopic to a simplicial map $g:\abs{S}\To \abs{\K}$, that is, $g$ is determined by the map on the underlying complexes, $\tilde g:S\To \K$. Thus, to show $\K$ is $d$-connected, it suffices to show that for a given simplicial map $f: S^n\To \K$ where $S^n$ is a simplicial $n$-sphere, there exists a simplicial $n$-ball, $B$, with $\dd B=S^n$, and a simplicial map $\f:B\To \K$ with $\f|_{\dd B}= f$. We now introduce some techniques we will apply to show such connectivity.

\begin{rem}\label{r:order}Let $S$ be a simplicial complex and $\K$ a \multi. Suppose $V(S)=A\sqcup B$, and we have given simplicial maps $f:S\cap A\To \K$ and $g:S\cap B\To \K$. On vertices $s\in S$, set
\begin{equation}\label{e:map}
    F(s)=\left\{
       \begin{array}{ll}
         f(s), & \hbox{if $s\in A$;} \\
         g(s), & \hbox{if $s\in B$.}
       \end{array}
     \right.
\end{equation}
If \eqref{e:map} defines a simplicial map $F:S\To \overline{\K}$, where $\overline{\K}$ is the simplicial complex underlying $\K$ (given by forgetting the ordering), then \eqref{e:map} also defines a simplicial map $F:S\To \K$, using the orderings provided by $f$ and $g$, and taking $A$ before $B$.
\end{rem}
\begin{defn}[Link move]\label{d:link} Let $S$ be a simplicial complex, $\K$ a \multi, and $f:S\To \K$ a simplicial map. Let $\s\in S$ be a simplex, and suppose we have a simplicial ball $B$ with $\dd B=\link(\s)$ and a simplicial map $\f:B\To \K$ with $\f|_{\dd B}=f|_{\link(\s)}$.

We have $\dd(\star(\s))= \link(\s)*\dd\s= \dd(B*\dd\s)$.
 Now replace $S$ by the simplicial $n$-manifold $S'=(S\fra \star(\s))\cup_{\link(\s)*\dd\s}B*\dd\s$. Also, replace $f$ by the simplicial map $f':S'\To \K$, which on a vertex $s\in S$ is given by
\begin{equation*}
    f'(s)=\left\{
             \begin{array}{ll}
              \f(s), & \hbox{if $\s\in B$;} \\
               f(s), & \hbox{if $\s\notin B$.}
             \end{array}
           \right.
\end{equation*}
and the ordering is given by $V(S')=V(B\fra \dd B)\sqcup V(S')\fra V(B\fra \dd B)$, as in Remark \ref{r:order}. Since both $\star(\s)$ and $B*\dd\s$ are faces of $B*\s$, there is a homotopy on on the geometric realizations from $\abs{f}$ to $\abs{f'}$. We call $f'$ the result of performing a link move to $f$ on $\s$ with $\f$.
\end{defn}

Now we return to $\L(g)$ and $\Ll(g)$. We define a crucial concept, which we call $\gcd$, that captures the essence of being a simplex in $\L(g)$.

\begin{defn}\label{d:ssc}For $A\del H$, we define $S(A)$, the smallest summand containing $A$, to be the summand $S(A)=\set{x\in H|\exists n\in \Z\fra \set{0}: nx\in \ideal{A}}$.
\end{defn}
\begin{defn}[$\gcd$]\label{d:gcd}Let $H$ be a free $\Z$ module, and $v_1,\ldots,v_n\in H$. We define $\gcd(v_1,\ldots,v_n)$ as follows. Let $V=\ideal{v_1,\ldots,v_n}$. If $\rank(V)<n$, then $\gcd(v_1,\ldots,v_n):=0$. If $\rank(V)=n$, let $S=S(V)$, and take a basis $(s_1,\ldots,s_n)$ of $S$. Let $A$ be the $n\times n$ matrix whose $i$th column is the coordinate vector of $v_i$ in the basis $(s_1,\ldots,s_n)$. Then $\gcd(v_1,\ldots, v_n)=\abs{\det(A)}$.

For a submodule $W\subset H$, we write $\gcd(W)$ for $\gcd(w_1,\ldots, w_k)$, where $w_1,\ldots, w_k$ is any basis of $W$.
\end{defn}
For a single vector $v$, $\gcd(v)$ is the greatest common divisor of the coefficients when writing $v$ in a basis for $H$, hence the name.

\begin{rem}\label{r:gcd2}Let $\vv=(v_0,\ldots,v_n)$, with $v_i\in H=H_1(F_{g,1};\Z)$. Then $\vv\in \L(g)$ if and only if $(v_0,\ldots,v_n)$ is isotropic, and $\gcd(v_0,\ldots,v_n)=1$.
\end{rem}

\begin{rem}\label{r:gcd} Use the partial order $\le_{\textup{div}}$ on $\N$, where $m\le_{\textup{div}}n$ means $m|n$. In particular $n\le_{\textup{div}} 0$ for all $n\in\N$. For two sets of vectors $V=\set{v_1,\ldots,v_n}$ and $W=\set{w_1,\ldots, w_m}$, we have $\gcd(V,W)\gediv \gcd(V)\gcd(W)$.
\end{rem}

\begin{rem}\label{r:gcdlig}If $H=A\oplus B$, and $A'\del A$, $B'\del B$ are subsets, then
\begin{equation*}
    \gcd(A',B')=\gcd(A')\cdot \gcd(B').
\end{equation*}
\end{rem}

In section \ref{ss:conntraels} we need the existence of dual vectors, as follows:
\begin{prop}\label{p:dual}Let $H$ be a free $\Z$-module of rank $2g$ with a symplectic form $\ialg{\cdot,\cdot}$. Given $n\le g$, and $v_1,\ldots,v_n\in H$, assume $\gcd(v_1,\ldots,v_n)=1$, and set $S=\ideal{v_1,\ldots,v_n}$.
\begin{itemize}
\item[$(i)$]There exists a \emph{dual summand} $D=D(v_1,\ldots,v_n)$ to $S$, meaning $D$ is an isotropic summand of rank $n$, and $D$ has a basis $u_1,\ldots,u_n$ satisfying
    \begin{equation*}
    \ialg{v_i,u_j}=\de_{ij}.
    \end{equation*}
    In particular, $S\oplus D$ is a symplectic summand, so there exists a unique summand $T\del H$, such that $H=(S\oplus D)\oplus T$ is a symplectic splitting.
\item[$(ii)$]Let $k\le g-n$. Let $D=D(v_1,\ldots,v_n)$ as in $(i)$ be given. Given vectors $w_1,\ldots,w_k$ with $\gcd(w_1,\ldots,w_k,S)=1$, there exists a dual summand $D_2$ of $S_2=\ideal{v_1,\ldots,v_n,w_1,\ldots,w_k}$ such that $S\oplus D\del S_2\oplus D_2$.
\item[$(iii)$]Let $k\le g-n$. Let $D=D(v_1,\ldots,v_n)$ and $T$ as in $(i)$ be given. Given vectors $w_1,\ldots,w_k$ with $\gcd(w_1,\ldots,w_k,S,D)=1$, there exist a dual summand $D(w_1,\ldots, w_k)\del T$.
\item[$(iv)$]Let $k\le g-n-m$. Let $S_1,D_1,T_1$ as in $(i)$, and let $S_2=\ideal{v_1,\ldots,v_n}$ with dual summand $D_2\del T_1$, also as in $(i)$. Then given $w_1,\ldots,w_k$ with $\gcd(w_1,\ldots,w_k,S_1,D_1,S_2)=1$, there exists a dual summand $D_3= D(v_1,\ldots,v_n,w_1,\ldots,w_k)\del T_1$ with $S_2\oplus D_2\del S_1 \oplus D_1 \oplus S_3\oplus D_3$.
\end{itemize}
\end{prop}
\begin{proof} $(i):$ We prove this by induction in $n$. It it not hard to see that one can 
%First 
find $u_1\in H$ with $\ialg{v_1,u_1}=1$ and $\gcd(u_1,v_1\ldots,v_n)=1$. 

Now consider $H_1=\ideal{v_1,u_1}^{\perp}$ which gives a symplectic splitting $H=\ideal{v_1,u_1}\oplus H_1$. For $n>1$, let $\tilde v_i=\pr_{H_1}(v_i)$ for $i=2,\ldots,n$. Note that $\gcd(\tilde v_2,\ldots,\tilde v_n)\lediv \gcd(u_1,v_1,v_2,\ldots,v_n)=1$. Then by induction we obtain $u_2,\ldots,u_n$ and $T$ satisfying the desired properties w.r.t. $\tilde v_2,\ldots,\tilde v_n$ in $H_1$. We change $u_1$ to $\overline u_1=u_1-\sum_{j=2}^n c_ju_j$, where $c_j=\ialg{v_i,u_1}$. One checks that $D=\ideal{\streg u_1,u_2,\ldots, u_n}$ is the desired dual summand. 

$(ii):$ We have $S,D,T$ given w.r.t. $v_1,\ldots,v_n$, as in $(i)$. First we claim there is a rank $2k$ symplectic summand $W$ in $T$, such that $S\oplus D\oplus W$ contains $w_1,\ldots,w_k$. To see this, consider the $k$ vectors $\tilde w_j=\pr_{T}(w_j)$, and take the smallest summand $S_W$ containing them. By $(i)$ we obtain $D_W,T_W$, where $S_W\oplus D_W$ is a symplectic summand of rank $\le 2k$. If the rank is $<2k$, add a symplectic summand $R_W\del T_W$, such that $W=S_W\oplus D_W\oplus R_W$ has rank $2k$.
% Now $S\oplus D\oplus W$ contains $w_1,\ldots,w_k$. 
Next, use $(i)$ on the vectors $v_1,\ldots,v_n,w_1,\ldots,w_k$ inside $S\oplus D\oplus W$, yielding a dual summand $S_2=S(v_1,\ldots,v_n,w_1,\ldots, w_k)$, and $T_2$. Note for dimensional reasons $T_2=0$. Thus $S\oplus D\del S\oplus D\oplus W=S_2\oplus D_2$.

$(iii)$ and $(iv)$ follow easily from $(i)$ and $(ii)$ as we now sketch: For $(iii)$, we have $S,D,T$ as in $(i)$. Use $(i)$ on $\tilde w_j=\pr_T(w_j)$ for $j=1,\ldots, k$, obtaining $D(\tilde w_1, \ldots, \tilde w_k)\del T$. Check this is a dual summand of $(w_1,\ldots,w_k)$. For $(iv)$, 
project $(v_1,\ldots,v_n, w_1,\ldots, w_k)$
on $T_1$; call the result $(\tilde{v}_1,\ldots,\tilde{v}_n,\tilde{w}_1,\ldots, \tilde{w}_k)$. Use $(ii)$ on $(\tilde{w}_1,\ldots, \tilde{w}_k)$, given $\tilde S_2=\ideal{\tilde{v}_1,\ldots,\tilde{v}_n}$ and $D_2$ in $T_1$, to get $D_3=D(\tilde{v}_1,\ldots,\tilde{v}_n,\tilde{w}_1,\ldots, \tilde{w}_k)$ with $\tilde{S}_2 \oplus D_2\del \tilde{S}_3 \oplus D_3\del T_1$. Check $D_3$ works.
\end{proof}

\begin{rem}Let $n\le g$. Given $v_1,\ldots,v_n$ in $H$ with $\gcd(v_1,\ldots,v_n)>0$, let $S=S(v_1,\ldots,v_n)$ denote the smallest summand containing $\ideal{v_1,\ldots,v_n}$. Then we can choose a basis $v_1',\ldots,v_n'$ for $S$ and get a dual summand $D=D(v_1',\ldots,v_n')$. We will call $D$ a dual summand of $S$ (w.r.t $v_1',\ldots,v_n'$).
\end{rem}

\subsection{Connectivity of $B(F_{g,1};1)/\I_{g,1}$}\label{s:con(F,1)}
In this section, $H=H_1(F_{g,1};\Z)$. We prove Theorem \ref{t:intro} for $i=1$; namely, the quotient complex $B(F_{g,1};1)/\I_{g,1}$ is $(g-2)$-connected. By Prop. \ref{p:complex} $(i)$, we must show
\begin{prop}\label{p:conn1a} $\L(g)$ is $(g-2)$-connected.
\end{prop}

This follows from Prop. \ref{p:conn1} by taking $\D^k=\Ø$ in $(ii)$. First we define:
\begin{defn}Let $\D^k\in \L(g)$ be a (possibly empty) simplex, and let $W\del H$ be a $\Z$-linear subspace. Write $\Letdel(g)=\link_{\L(g)}(\D^k)$, and define $\Letblank{\D^k;W}(g)$ to be the the subcomplex of $\Letdel(g)$ consisting of the simplices whose vertices are in $W$.
\end{defn}
The proof of the following proposition is modeled on \cite{Putman2} Prop. 6.13, but his argument has a gap, which we repair.
\begin{prop}\label{p:conn1} For $g\ge 1$, fix $-1 \le k< g$. Let $\D^k$ be a $k$-simplex in $\L(g)$. Then for all vectors $x\in \Letdel(g)$, the following hold.
\begin{itemize}
  \item[$(i)$] For $-1 \le n \le g-k-3$, we have $\pi_n(\Letdelx(g)) = 0$.
  \item[$(ii)$] For $-1 \le n \le g-k-3$, we have $\pi_n(\L^{\D^k}(g)) = 0$.
\end{itemize}
\end{prop}
\begin{proof}We first prove $(i)$. Assume inductively that $\pi_{n'}(\Letblankx{\D^{k'}}(g)) = 0$ and $\pi_{n'}(\L^{\D_{k'}}(g)) = 0$ for all $n'<n$ and all $\D^{k'}$ such that $n'\le g-k'-3$. The case $n=-1$ holds, since $x\in\Letdelx(g)\neq \emptyset$ for $k<g$.

So let $n\ge 0$, and let $S$ be a simplicial $n$-sphere and $f: S \To \Letdelx(g)$ a simplicial map. %By Lemma \ref{l:simpapprox}, it is enough to show that $\f$ is homotopic to a constant map.
Fix a symplectic basis $\mathfrak{X}$ of $H$ extending the isotropic basis $(\D^k ,x)$, use this basis to define $\rk^{x}$ as in Def. \ref{d:rank}, and consider
\begin{equation}\label{e:R}
    R=R_x=\max\set{\abs{\rk^{x}(\f(s))}\mid s\in S^{(0)}}.
\end{equation}

If $R=0$, then $f(S)\del \link_{\Letdelx}(x)$, and we can define a simplicial map $F:B\To \Letdelx(g)$ where $B=S*{+}$, by $F(+)=x$, as in Remark \ref{r:order}.

Now assume that $R>0$ and call $\s\in S$ regular bad if all vertices $s$ of $\s$ satisfy $\abs{\rk^{x}(\f(s))}=R$. Let $\s$ be a regular bad simplex of maximal dimension, say $\dim\s=m$. By maximality of $\s$ we get
\begin{equation}\label{e:inductmap}
    f|_{\link(\s)}:\link_S(\s)\To\Letblankx{\D^k*f(\s)}(g).
\end{equation}
Here, $\link_S(\s)$ is a simplicial $(n-m-1)$-sphere, and the goal is to obtain a simplicial $(n-m)$-ball $B$ with $\dd B=\link_S(\s)$ and a simplicial map 
\begin{equation}\label{e:fgoal}
	\f:B\To \Letblank{\D^k*f(\s);\xort}(g),\quad \text{with }\f|_{\dd B}=f|_{\link(\s)}.
\end{equation}
This follows from the inductive hypothesis if $x\in \Letblank{\D^k*f(\s)}(g)$. But this might not be the case, so assume $x\notin \Letblank{\D^k*f(\s)}(g)$, in other words $\gcd(x,f(\s),\D^k)\neq 1$. (This is what is missing in Putman's argument). 
 
There are two possibilities. The first is $\gcd(x,f(\s),\D^k)>1$. In this case, the smallest summand $V$ containing $\ideal{x,f(\s),\D^k}$ has rank $1+(\dim f(\s)+1)+(k+1)$, and so we can choose a basis for $V$ of the form $\set{\tilde x, f(\s),\D^k}$. Since $V$ is isotropic, we get
\begin{equation*}
    \Letblankx{\D^k*f(\s)}(g)=\Letblank{\D^k*f(\s);\ideal{\tilde x}^{\perp}}(g).
\end{equation*}
Now by construction, $\tilde x\in \Letblank{\D^k*f(\s)}(g)$, so we get \eqref{e:fgoal} by induction.

The second possibility is $\gcd(x,f(\s),\D^k)= 0$. Then $V=\ideal{f(\s),\D^k}$ is a summand, and $x\in V$. Choose a basis of $V$ extending $x$, i.e. $\set{x,b_0,\ldots,b_\ell}$, such that $\rk^x(b_i)=0$. Note $\ell=\dim(f(\s))+k$, and $\ww=(b_0,\ldots,b_\ell)\in \L(g)$. Then
\begin{equation}\label{e:linklig}
   \Letblankx{\D^k*f(\s)}(g)=\Letblankx{x*\ww}(g)
\end{equation}
Let $y$ denote the basis vector in $\mathfrak{X}$ dual to $x$, i.e. $\ialg{x,y}=1$. Consider
\begin{equation}\label{e:proj}
    \pr: \xort\To \ideal{x,y}^{\perp},\qquad \pr(h)=h-\ialg{h,y}x.
\end{equation}
This can be extended to a map on simplices, which we call $\pr$ again, by using $\pr$ on each vertex. Then for $\vv\in \Letblank{x*\ww,\xort}(g)$ we get that $\pr(\vv)\in \Letblank{x*\ww,\ideal{x,y}^{\perp}}(g)$, from \eqref{e:proj}. We can identify $\ideal{x,y}^{\perp}$ with $\ideal{a_1,b_1,\ldots,a_{g-1},b_{g-1}}$, and since $b_0,\ldots,b_\ell \in \ideal{x,y}^{\perp}$, this identification turns $\pr$ into a map
\begin{equation}\label{e:pr}
    \pr: \Letblank{x*\ww,\xort}(g)\To \Letblank{\ww}(g-1).
\end{equation}
We then consider the composition $\pr\circ f|_{\link(\s)}$, and get by induction in $(ii)$ that there is a simplicial ball $B$ with $\dd B=\link(\s)$ and a simplicial map $\tilde\f$ such that the left-hand square commutes in the following diagram 
\begin{equation}\label{e:phi}
\xymatrix{
	\link(\s)\ar[r]^{f\quad}\ar[d] & \Letblank{\D^k*f(\s);\xort}(g)\ar@{=}[r] & \Letblank{x*\ww;\xort}(g)\ar[dl]_{\pr} \\
	B\ar[r]^{\tilde \f\quad} & \Letblank{\ww}(g-1)\ar[r]^{\psi}_{\iso} & \Letblank{x*\ww; \ortZ{x,y}}\ar[u]_j
	}
\end{equation}
Here, $j$ is induced by the subspace inclusion $\xyort\into H$, and $\pr\circ j\circ \psi=\id$. We modify $\tilde\f$ to a map $\f: B\To \Letblank{\D^k*f(\s)}(g)$ satisfying \eqref{e:fgoal} via 
	\[
	\f(s)=\left\{
      \begin{array}{ll}
        f(s)   & \hbox{if }s\in \link(\s), \\
   j\circ\psi\circ \tilde \f(s)& \hbox{if }s\in B\fra \dd B.
      \end{array}
    \right.
\]
and Remark \ref{r:order}. (To show $\f$ is well-defined, use $\tilde\f=\pr\circ \f$.)

This shows we have $\f$ as in \eqref{e:fgoal}. We now modify $\f$ to a map $\f'$ by performing division with remainder, as in \cite{Putman2}: Let $t\in \s$ be fixed, set $v=\f(t)$. By division we obtain $q_s\in\Z$ such that $\abs{\rk^x(\f(s)-q_sv)}<\abs{\rk^x(v)}=R$ for all $s\in B^{(0)}$. For $s\in\dd B=\link(\s)$ we take $q_s=0$. We then set $\f'(s)=\f(s)-q_sv$ for $s\in B^{(0)}$. By Remark \ref{r:order} we get a simplicial map $\f':B\To \Letblank{\D^k*f(\s),\xort}(g)$ with $\abs{\rk^x(\f'(s))}<R$ for all $s\in B^{(0)}$. Then we do a link move to $f$ on $\s$ with $\f'$ (see Def. \ref{d:link}), which produces a map homotopic to $f$, removing $\s$. Continuing this process inductively in the maximal dimension of regular bad simplices, we can obtain $R=0$, so we are done.

We next prove $(ii)$. This is done in a similar manner, but instead of $R_x$ we use $R_y$, where again $y$ is the dual basis vector to $x$. For $R_y=0$ we are in case $(i)$ and we are done. For $R_y>0$, we remove bad simplices precisely as above, which is easier since the analogue of \eqref{e:inductmap} now directly implies \eqref{e:fgoal}.
\end{proof}

\subsection{Connectivity of $B(F_{g,2};2)/\I_{g,2}$, first part}\label{ss:conn2}
In this section, $H=H(g+1)=H_1(F_{g+1,1};\Z)$. We prove Theorem \ref{t:intro} for $i=2$: The quotient complex $B(F_{g,2};2)/\I_{g,2}$ is $(g-2)$-connected. By Prop. \ref{p:complex} $(ii)$, to prove this we must show:
\begin{thm}\label{p:conn2}$\Ll(g)$ is $(g-2)$-connected.
\end{thm}

\begin{defn}For a vector $v\in H$, let $\pr_2(v)$ denote the projection of $v$ onto the subspace $\ideal{a_2,b_2,\ldots, a_{g+1},b_{g+1}}$. For a simplex $\vv=(v_1,\ldots,v_n)\in\Ll(g)$, let $\pr_2(\vv)=(\pr_2(v_1),\ldots,\pr_2(v_n))$.
\end{defn}

The basic idea behind the proof of connectivity of $\Ll(g)$ is the following: Because the $a_1$-coordinate in a simplex $\vv\in\Ll(g)$ is fixed to be $1$, we cannot manipulate the vectors of $\vv$ as we did by using division with remainder in the proof for $\L(g)$. We take two major steps be able to ignore the $a_1$- and $b_1$-coordinates of $\vv$: In section \ref{ss:conn2}, we reduce to the case where the $b_1$-coordinate is fixed, and the rest, $\pr_2(\vv)$, form a simplex in $\L(g)$. Section \ref{ss:conntraels} is then dedicated to adapting the proof of Prop. \ref{p:conn1} to the new situation.

For a simplex $\vv\in\Ll(g)$, we will often need the projection map $\pr_2$ in connection with $\gcd$ (see Def. \ref{d:gcd}), so we introduce the following notation:
\begin{defn}
$\gcdto(\vv)=\gcd(\pr_2(\vv))$.
\end{defn}
We recall from Remark \ref{r:gcd} if $\vv,\ww\in \Ll(g)$ and $\vv*\ww$ is a simplex, then
\begin{equation}\label{e:gcd}
\gcdto(\vv*\ww)\gediv \gcdto(\vv)\gcdto(\ww),
\end{equation}
From now on, for $\vv\in \Ll(g)$ (sections \ref{ss:conn2}, \ref{ss:conntraels}, and \ref{ss:finishexact2}) we write $S(\vv)=S(\pr_2(\vv))$, see Def. \ref{d:ssc}.
\begin{defn}\label{d:gode}Let $\D^k \in \Ll(g)$. Write $\Ldel(g)=\link_{\Ll(g)}(\D^k)$.
\begin{itemize}
%\item Let $\Ldel(g)=\link_{\Ll(g)}(\D^k)$ be the link of $\D^k$ in $\Ll(g)$.
\item Let $\Ldelpos(g)$ be the subcomplex of $\Ll(g)$ consisting of simplices $\vv$ satisfying $\gcdto(\vv)\neq 0$.
\item Let $\Ldelgode(g)$ be the subcomplex of $\Ldel(g)$ consisting of simplices $\vv$ satisfying $\gcdto(\vv;S(\D^k))=1$.
\item Let $t\in\Z$. Define $\Ldelgodeb(g)$ to be the subcomplex of $\Ldelgode(g)$ consisting of simplices $(v_1,\ldots,v_n)$ where $\rk^{b_1}(v_i)=t$ for all $i$.
\end{itemize}
\end{defn}

\begin{rem}\label{r:gcdnonzero}If $\gcdto(\vv;S(\D^k))=1$, then the inequality \eqref{e:gcd} implies that $\gcdto(\ww; S(\D^j))=1$ for all subsimplices $\D^j\del \D^k$ and all subsimplices $\ww\del \vv$. \end{rem}

We first consider what happens when $\gcdto(\D^k)=0$.

\begin{lem}\label{l:gcd0conn}
Let $\D^k$ be a $k$-simplex in $\Ll(g)$ with $\gcdto(\D^k)=0$. Then $\Ldel(g)$ is $(g-k-2)$-connected.
\end{lem}
\begin{proof}
Let $\D^k=(v_0,\ldots,v_k)$ and denote $\pr_2(\D^k)$ by $(\tilde v_0,\ldots, \tilde v_k)$, i.e. $v_i=a_1+r_ib_1+ \tilde v_i$, $i=0,\ldots,k$.  Since $\gcdto(\D^k)=0$, the set $\set{\tilde v_0,\ldots, \tilde v_k}$ is linearly dependent, which gives some $c_0,\ldots,c_k\in\Z$ relatively prime, with
\begin{equation}\label{e:depend}
\sum_{i=0}^kc_i v_i= sa_1+tb_1, \quad \text{for some }s,t\in\Z.
\end{equation}
Since $\set{v_0,\ldots,v_k}$ is isotropic, $\ialg{v_i,sa_1+tb_1}=0$, meaning $t-sr_i=0$ for all $i=0,\ldots,k$. So $s\mid t$, and we can assume $s=1$. Therefore, $r_i=\rk^{b_1}(v_i)=t$ for $i=0,\ldots,k$. 
Using \eqref{e:depend} it is easy to conclude that for any $\ww\in\Ldel(g)$, the $b_1$-coordinate of each vertex in $\ww$ is always $t$.

Write $H_2:=\pr_2(H)=\ideal{a_2,b_2,\ldots,a_{g+1},b_{g+1}}$. Let $(\tilde x_1,\ldots,\tilde x_k)$ be a basis of $\ideal{\tilde v_0,\ldots,\tilde v_k}$ in $H_2$; then set $x_0=a_1+tb_1$, and $x_i=x_0+\tilde x_i$. Then $\Lambda^k=(x_0,x_1,\ldots,x_k)$ is also a simplex in $\Ll$, and  $\Lblank{\Lambda^k}(g)=\Ldel(g)$. 

If we identify $H_2$ with $H(g)$, then we see that $\tilde\Lambda^{k-1}:=(\tilde x_1,\ldots,\tilde x_k)$ becomes a $(k-1)$-simplex in $\L(g)$. Then $\Lblank{\Lambda^k}(g)\iso \Letblank{\tilde \Lambda^{k-1}}(g)$ via the isomorphism
$ \vv\mapsto\pr_2(\vv)$, since $\vv=a_1+tb_1+\pr_2(\vv)$ by the above. From Prop. \ref{p:conn1} we know $\L^{\tilde \Lambda^{k-1}}(g)$ is $(g-k-2)$-connected.
\end{proof}

Consequently, we can focus on simplices $\D^k\in \Ll(g)$ with $\gcdto(\D^k)\neq 0$:
\begin{prop}\label{p:reduce}Let $g\ge 3$ and $-1\le k\le g-1$. For any $k$-simplex $\D^k\in \Ll(g)$ with $\gcdto(\D^k)\neq 0$, consider the following:
\begin{itemize}
  \item[$(i)$] $\pi_n(\Ll(g))=0$ for $-1\le n\le g-2$,
  \item[$(ii)$] $\pi_n(\Ldelpos(g))=0$ for $-1\le n\le g-2$,
  \item[$(iii)$] $\pi_n(\Ldelgode(g))=0$ for $-1\le n\le g-k-3$,
  \item[$(iv)$] $\pi_n(\Ldelgodeb(g))=0$ for $-1\le n\le g-k-3$, where 
\begin{equation}\label{e:t}
	t= \left\{	
\begin{array}{ll}
	0&\hbox{if }\gcdto(\D^k)=1,\\
	\rk^{b_1}(v_0)&\hbox{if }\gcdto(\D^k)\neq1, \hbox{where }\D^k=(v_0,\ldots,v_k).
\end{array}
 \right.
\end{equation}
\end{itemize}
Then $(iv)\Tto(iii)\Tto(ii)\Tto(i)$.
\end{prop}
\begin{proof} All the implications $\Tto$ will be shown similarly, so we give the first one in detail, and in the others focus on the differences.

$(iv)\Tto (iii)$: Assume $-1\le n\le g-k-3$. Let $S$ be a simplicial $n$-sphere, and let $f:S\To \Ldelgode(g)$ be a simplicial map. We wish to homotope $f$ and $S$ so $f(S)$ lies in $\Ldelgodeb(g)$.

Let $t$ be as specified in \eqref{e:t}. Call a simplex $\s\in S$ \emph{regular bad}, if all $b_1$-coordinates in $f(\s)$ are $\neq t$. Let $\s\in S$ be regular bad of maximal dimension, say $\dim(\s)=m$.

We claim $f(\link(\s))\del \link(f(\s))$. Since $f$ is simplicial, it suffices to show that $f(\s)\cap f(\link(\s))=\Ø$. This follows from regularly bad of maximal dimension; indeed if not, and $v\in f(\s)\cap f(\link(\s))$ is a vertex, then $v=f(s)$ for $s\in S^{(0)}$, and $s*\s$ would also be regular bad, contradicting the maximality of $\s$. This argument is quite general (it holds for most definitions of regular bad we will use) and the result will henceforth be used without comment.

It follows that every simplex $\vv\in f(\link(\s))$ has the property that all $b_1$-coordinates of $\vv$ are $t$. So 
\begin{equation}\label{e:flink}
	f|_{\link(\s)}:\link(\s)\To \Lblankgodeb{f(\s)*\D^k}(g),
\end{equation}
and we know from $(iv)$, since $\dim(f(\s)*\D^k)\le k+m+1$,
that $\Lblankgodeb{f(\s)*\D^k}(g)$ is $(g-k-m-2)$-connected. Also, $\link(\s)$ is an $(n-m-1)$-sphere, where $n-m-1\le g-k-m-2$. So there is a simplicial $(n-m)$-ball $B$ with $\dd B=\link(\tau)$, and a map $\f: B\To  \Lblank{\D^k*\ww}(g)$, such that $\f|_{\dd B}= f|_{\link(\tau)}$. Now we perform a link move to $f$ on $\tau$ with $\f$. Call the resulting map $f'$; it is homotopic to $f$. Note, all this follows from \eqref{e:flink} and the induction in $(iv)$.

We wish to show that we have introduced no new regular bad simplices in $S'$ of dimension $\ge m$. By construction a new simplex in $S'$ has the form $\tau_1*\tau_2$, where $\tau_1\in\dd\s$ and $\tau_2\in B$ (one of them can be the empty simplex). Thus $f(\tau_2)$ has all $b_1$-coordinates equal to $t$, and so if $\tau_2\neq \Ø$, then $\tau_1*\tau_2$ cannot be regularly bad. So we have introduced no new regular bad simplices.

This shows we can, through homotopies of the starting map $f$, remove all bad simplices by induction in the maximal dimension of regular bad simplices. When there are no regular bad simplices left, we have $f:S\To \Ldelgodeb(g)$, and by $(iv)$ this complex is $(g-k-3)$-connected, so we are done.

\vspace{0.3cm}\noindent $(iii)\Tto (ii)$:  Let $S$ be a simplicial $n$-sphere, and let $f:S\To \Ldelpos(g)$ be a simplicial map. We say $\s\in S$ is regular bad if for all vertices $v\in f(\s)$ we have $\gcdto(v, S(f(\s)\fra v))>1$. Here $f(\s)\fra v$ denotes the difference in vertex sets. Let $\s\in S$ be regular bad of maximal dimension.%, say $\dim(\s)=m$.

We claim:
      \begin{equation}\label{e:claim}
      f|_{\link(\s)}:\link(\s) \To \Lblankgode{f(\s)}(g).
      \end{equation}
By maximality, $f(\link(\s))\del \link(f(\s))$. So we must show for all $\tau\del \link(\s)$ that
$\gcdto(f(\tau),S(f(\s)))=1$. Assume for contradiction there is $\tau\del \link(\s)$ such that $\gcdto(f(\tau),S(f(\s)))>1$.
We know $\t*\s$ is not regular bad by maximality of $\s$, so there is a vertex $v\in f(\t)*f(\s)$, such that
\begin{equation}\label{e:fra}
\gcdto(v, S(f(\tau)*f(\s)\fra v))=1.
\end{equation}
If $v\in f(\s)$ then we get by Remark \ref{r:gcdnonzero}:
\begin{equation*}
1=\gcdto(v, S(f(\t)*f(\s)\fra v))=\gcdto(v, S(f(\s)\fra v))> 1.
\end{equation*}
So we know that $v\in f(\t)$. Consider $f(\t)\fra v$. We see from \eqref{e:fra} that
\begin{equation*}
\gcdto(f(\t)\fra v, S(f(\s)))= \gcdto(v, f(\t)\fra v, S(f(\s)))=\gcdto(f(\t), S(f(\s)))\neq 1
\end{equation*}
Thus we can use the same argument with $f(\t)\fra v$ instead of $f(\t)$. Iterating this, we reach the absurd conclusion that
$\gcdto(S(f(\s)))\ne 1$, so we have shown the claim \eqref{e:claim}.

Now the proof runs as above by induction in $(iii)$.  When there are no regular bad simplices left, we have $f:S\To \Ldelgode(g)$, so we are done.

\vspace{0.3cm}\noindent $(ii)\Tto(i)$: A simplex $\s\in S$ is called regular bad if it satisfies both $\gcdto(f(\s))=0$, and $\gcdto(\vv)\neq 0$ for all proper subsimplices $\vv\subsetneq f(\s)$. Let $\s$ be regular bad of maximal dimension, say $\dim(\s)=m$. Then $f(\link(\s))\del \link(f(\s))$, and by Lemma \ref{l:gcd0conn},  $\link_{\Ll(g)}(f(\s))=\Lblank{f(\s)}(g)$ is at least $(g-m-2)$-connected. Using this instead of induction yields the result.
\end{proof}

\subsection{Connectivity of $B(F_{g,2};2)/\I_{g,2}$, second part}\label{ss:conntraels}

In this section we prove the connectivity of $\Ldelgodeb(g)$, where $t\in\Z$ is as in \eqref{e:t}. This turns out to be trickier than one should think, and we need more reductions to prove the result. The problem is that $\D^k$ itself need neither satisfy $\gcdto(\D^k)=1$ nor that $\rk^{b_1}(v)=t$ for all vertices $v$ of $\D^k$.

In this section, recall the meaning of $\D=\D_1*\D_2$ etc, from Remark \ref{r:join}.

\begin{rem}\label{r:dual}We will apply Prop. \ref{p:dual} to the projection simplices, and use the following notation: If $\D$ is an $n$-simplex, we will write $S(\D)=S(\pr_2(\D))$, the smallest summand containing $\pr_2(\D)$. Then $D(\D)$ denotes a dual summand of $S(\D)$ in $H_2=\pr_2(H)$, and $T(\D)$ the symplectic subspace such that $S(\D)\oplus D(\D)\oplus T(\D)=H_2$. Then $(iii)$ and $(iv)$ of Prop. \ref{p:dual} can be stated as follows:
\begin{itemize}
    \item[$(i)$]Given $\D_1$ and $\D_2$ such that $\D_1*\D_2$ is a simplex with $\gcdto(\D_1*\D_2)\neq 0$, and given a dual summand $D(\D_1)$. If $\gcdto(\D_2, S(\D_1),D(\D_1))=1$, then there is $D(\D_2)\del T(\D_1)$. In particular we can choose $D(\D_1*\D_2)=D(\D_1)\oplus D(\D_2)$.
\item[$(ii)$]Given $\D_1$,$\D_2$ and $\D_3$ such that $\D_1*\D_2*\D_3$ is a simplex, and given dual summands $D(\D_1)$ and $D(\D_2)$. If $\gcdto(\D_2*\D_3, S(\D_1),D(\D_1))=1$, then there is $D(\D_2*\D_3)\del T(\D_1)$ with
\begin{equation*}
    S(\D_2)\oplus D(\D_2)\del S(\D_1) \oplus D(\D_1) \oplus S(\D_2*\D_3)\oplus D(\D_2*\D_3)
\end{equation*}
\end{itemize}
\end{rem}

\begin{defn}
Let $\D=\D_1*\D_2*\D_3\in \Ldelpos(g)$, and assume that $\gcdto(\D_2, S(\D_1),D(\D_1))=1$. Let $D(\D_1)$ and $D(\D_2)\del T(\D_1)$ denote a choice of dual summands of $S(\D_1)$ and $S(\D_2)$, respectively, as in Remark \ref{r:dual} $(i)$. 
We define $\Mdel(g)$ to be the subcomplex of $\Lblankgodeb{\D}(g)$ consisting of simplices $\ww$ which satisfy:
\begin{itemize}
  \item[(a)]$\gcd_2\Big(\ww,S\big(\D_1,\D_2,\D_3,D(\D_1),D(\D_2)\big)\Big)=1$.
  \item[(b)] $\ww \perp D(\D_1)$.
\end{itemize}
%Note: If we write $S(\D_i)=\Ø$, it is understood that $\D_i=0$.
\end{defn}
The reader should be aware that the role of the \emph{first} non-empty simplex among $\D_1,\D_2,\D_3$ is to be a bad simplex from Prop. \ref{p:reduce}, so we can only assume it is in $\Ldelpos(g)$. 

\begin{rem}\label{r:idea}This is the idea: First note that $\Mdelblank{0}{0}{\Ø,\Ø,\D}(g)=\Lblankgodeb{\D}(g)$, which we need to show is $(g-k-3)$-connected. The following proposition reduces this to showing that $\Mdelblank{D(\D)}{0}{\D,\Ø,\Ø}(g)$ is $(g-k-3)$-connected, and in this complex, it is possible to make modifications enough to do division with remainder, as in Prop. \ref{p:putmansver}: Indeed, if we set $\tilde f(s)=\f(s)-q_sv$, where $v\in f(\s)$, then $\rk^{a_1}(\tilde f(s)\neq 1$. To remedy this, we are forced to use $\tilde\f'(s)=a_1+tb_1+\pr_2(\f(s)-q_sv)$ instead, but then $\tilde\f'(s)$ is no longer orthogonal to $\D$. And here $\Mdelblank{D(\D)}{0}{\D}(g)$ saves the day: All its simplices are orthogonal to $D(\D)$, so we can set $\f(s)=\tilde\f'(s)+u_s$, where $u_s\in D(\D)$ satisfies $\f(s)\perp \D$, without changing anything else ($u_s$ is constructed in Lemma \ref{r:compensate}).
\end{rem}
\begin{lem}\label{l:reduce2}Let $\D=\D_1*\D_2*\D_3\in \Ldelpos(g)$ be a $k$-simplex. Assume: 
\begin{itemize}
	\item If $\D_1\neq\Ø$, that $\D_1\in\Ldelpos(g)$, $\D_2\in \Mdelblank{0}{D(\D_1)}{\Ø,\D_1,\Ø}(g)$, $\D_3\in \Mdelblank{D(\D_1)}{0}{\D_1,\Ø,\D_2}(g)$. 
	\item If $\D_1=\Ø$, that $\D_2\in\Ldelpos(g)$ and $\D_3\in \Mdelblank{0}{0}{\Ø,\Ø,\D_2}(g)=\Lblankgodeb{\D_2}(g)$. 
	\item If $\D_1=\D_2=\Ø$, that $\D_3\in\Ldelpos(g)$.
\end{itemize}
  Consider the following:
\begin{itemize}
  \item[$(i)$] $\pi_n(\Mdel(g))=0$ for $n\le g-k-3$.
	\item[$(ii)$] $\pi_n(\Mdelblank{D(\D_1)}{D(\D_2)}{\D_1,\D_2,\Ø}(g))=0$ for $n\le g-k-3$.
  \item[$(iii)$] $\pi_n(\Mdelblank{D(\D_1)}{0}{\D_1,\Ø,\Ø}(g))=0$ for $n\le g-k-3$.
\end{itemize}
Then $(iii)$ implies $(i)$ and $(ii)$.
\end{lem}
\begin{rem}\label{r:delcomplex}By the assumptions in the Lemma, one checks that Remark \ref{r:dual} can be used to create new dual summands, thereby ensuring that
\begin{eqnarray*}
  \Mdelblank{D(\D_1)}{D(\D_2*\D_3)}{\D_1,\D_2*\D_3,\Ø}(g) &\del& \Mdel(g), \\
  \Mdelblank{D(\D_1*\D_2)}{0}{\D_1*\D_2,\Ø,\Ø}(g) &\del&  \Mdelblank{D(\D_1)}{D(\D_2)}{\D_1,\D_2,\Ø}(g).
\end{eqnarray*}
\end{rem}

\begin{proof}We use the same strategy as the proof of Prop. \ref{p:reduce}. The argument is inductive in $n$, so let $n$ be fixed.

$(i):$ We inductively assume $(i)$ for all $n'<n$, and $(ii)$. Let $f:S\To \Mdel(g)$ be a simplicial map from a simplicial $n$-sphere $S$. Since $\D_2*\D_3\in \Mdelblank{0}{D(\D_1)}{\Ø,\D_1,\Ø}$, we can construct $D(\D_2*\D_3)$ as in Remark \ref{r:dual} $(ii)$.

We say $\s\in S$ is regular bad if for all vertexes $v\in f(\s)$, we have
\begin{equation}\label{e:regbad}
\gcdto\Big(v,S\big(f(\s)\fra v,\D,D(\D_1),D(\D_2*\D_3)\big)\Big)\ne 1,
\end{equation}
where $f(\s)\fra v$ is the difference between the vertex sets.  Let $\s$ be a regular bad simplex of maximal dimension. We claim, see Remark \ref{r:delcomplex}, that
      \begin{equation*}
      f|_{\link(\s)}:\link(\s) \To \Mdelblank{D(\D_1)}{D(\D_2*\D_3)}{\D_1,\D_2*\D_3,f(\s)}(g).
      \end{equation*}
To see this, we must show for all $\tau\del \link(\s)$ that
\begin{equation*}
\gcdto\Big(f(\tau),S\big(\D,f(\s),D(\D_1),D(\D_2*\D_3)\big)\Big)=1.
\end{equation*}
The argument is verbatim as in the proof of Prop. \ref{p:reduce} $(iii)\Tto (ii)$, replacing $S(f(\s))$ with $S(\D,f(\s),D(\D_1),D(\D_2*\D_3))$.
We can now use $(i)$ inductively to fill put the link in $\Mdelblank{D(\D_1)}{D(\D_2*\D_3)}{\D_1,\D_2*\D_3,f(\s)}(g)$, perform a link move to $f$, and check that this creates no new regular bad simplices, as in the proof of Prop. \ref{p:reduce}.

Performing this process inductively, we can assume that there are no regular bad simplices of $f$. Then by definition, $f(S)\del  \Mdelblank{D(\D_1)}{D(\D_2*\D_3)}{\D_1,\D_2*\D_3,\Ø}(g)$
and by $(ii)$, we are done.

$(ii):$ We say $\s\in S$ is regular bad if all vertices $s$ of $\s$ satisfy $s\not\perp S(\D_2)$. Let $\s$ be a regular bad simplex of maximal dimension. We can choose $D(\D_1*\D_2)=D(\D_1)\oplus D(\D_2)$ by Remark \ref{r:dual} $(i)$. Now use $(i)$ inductively on $\link(\s)$ to fill out the link in $\Mdelblank{D(\D_1)\oplus D(\D_2)}{0}{\D_1*\D_2,\Ø,f(\s)}(g)$. After removing all regular bad simplices, we are in case $(iii)$.
\end{proof}
We now construct the $u\in D(\D)$ mentioned in Remark \ref{r:idea}: 
\begin{lem}\label{r:compensate}Given $\D$, there exists $u\in D(\D)$ such that $a_1+tb_1+u\perp \D$.
\end{lem}
\begin{proof}Write $\D=(v_0,v_1,\ldots,v_k)$. There are two cases:

First if $\gcdto(\D)=1$, write $\pr_2(\D)=(v_0',\ldots,v_k')$ (this is a basis of $S(\D)$). Let $(u_0,\ldots, u_k)$ denote a dual basis. Set $u=\sum_{j=0}^k \ialg{v_j,a_1+tb_1}u_j$. Then $\ialg{u,v_j'}=\ialg{v_j,a_1+tb_1}$. Since $u\in D(\D)$, we get for $j=0,\ldots,k$,
\begin{equation*}
    \ialg{u,v_j}=\ialg{\pr_2(u),v_j}=\ialg{u,\pr_2(v_j)}= \ialg{u,v_j'} = -\ialg{a_1+tb_1,v_j}.
\end{equation*}

If $\gcdto(\D)>1$, then $t=\rk^{b_1}(v_0)$, see Prop. \ref{p:reduce} $(iv)$. Therefore $\ialg{a_1+tb_1,v_j}=-\ialg{\pr_2(v_0),v_j}$. Let $(v_0',\ldots,v_k')$ be a basis of $S(\D)$, and let $(u_0,\ldots,u_k)$ be a dual basis. Set $u= \sum_{j=0}^k\ialg{\pr_2(v_0),v_j'}u_j$. Then %$\ialg{u,v_i'}=\ialg{\pr_2(v_0),v_i'}$ for all $i$, which implies
$\ialg{u,v}=\ialg{\pr_2(v_0),v}$ for $v\in S(\D)$. Thus for all $j=0,\ldots,k$:
\begin{equation*}
    \ialg{u,v_j}= \ialg{u,\pr_2(v_j)}= \ialg{\pr_2(v_0),v_j}=-\ialg{a_1+tb_1,v_j}.
\end{equation*}
This shows the lemma.\end{proof}

Finally we can show the remaining part, $(iii)$ of Lemma \ref{l:reduce2}. We have given a $k$-simplex $\D\in \Ldelpos(g)$, a dual summand $D(\D)$, and $T(\D)$ as in Remark \ref{r:dual}. To ease the notation, let $\Ndel(g)=\Mdelblank{D(\D)}{0}{\D,\Ø,\Ø}(g)$, and similarly for $\ww\del\Ndel(g)$, let $\Ndelblank{\ww}(g)=\Mdelblank{D(\D)}{0}{\D,\Ø,\ww}(g)$ be the link of $\ww$ in $\Ndel(g)$. 

The proof will be similar to the proof of Prop. \ref{p:conn1}, and we define:
For $x\in T(\D)$ with $x\perp \ww$, let $\Ndelblankx{\ww}(g)$ be the subcomplex of $\Ndelblank{\ww}(g)$ consisting of all simplices whose vertices are in $\xort$. 

\begin{prop}\label{p:putmansver}
Let $\D\in \Ldelpos(g)$ be a $k$-simplex, and $\ww\in \Ndel(g)$ an $m$-simplex. Let $x\in T(\D)$ with $x\perp \ww$ and $\gcdto(x,\ww,S(\D),D(\D))=1$.  Then
\begin{itemize}
  \item[$(i)$]$\pi_n(\Ndelblankx{\ww}(g))=0$ for $n\le g-k-m-4$.
  \item[$(ii)$]$\pi_n(\Ndelblank{\ww}(g))=0$ for $n\le g-k-m-4$.
\end{itemize}
\end{prop}

\begin{proof}We prove $(i)$ inductively, assuming both $(i)$ and $(ii)$ for all $n'<n$ and all $k'$, $m'$ such that $n'\le g-k'-m'-4$. We have given a simplicial $n$-sphere $S$ and a simplicial map $f: S\To \Ndelblankx{\ww}(g)$. Fix a symplectic basis $\mathfrak{X}$ for $H$ extending $x\in T$ with the dual  basis vector $y$ to $x$ also satisfying $y\in T$. Define $R=R_x$ as in \eqref{e:R}.

If $R=0$, let $u\in D(\D)$ with $a_1+tb_1+u\perp \D$ be the element from Lemma \ref{r:compensate}. Then proceed as in Prop. \ref{p:conn1}, except $F(+)=a_1+tb_1+x+u$.

If $R>0$: Call a simplex $\s$ in $S$ regular bad if $\abs{\rk^x(f(s))}=R$ for all vertices $s\in\s$. Let $\s$ be a regular bad simplex of maximal dimension. Then
\begin{equation}\label{e:induction}
    f|_{\link(\s)}:\link(\s) \To \Ndelblankx{\ww*f(\s)}(g).
\end{equation}
But we cannot be sure $x$ satisfies $\gcdto(x,\ww*f(\s),S(\D),D(\D))=1$. If not, there are two possibilities. To ease the notation write $\ww'=\ww*f(\s)$.

We will need the following observations time and again: For $h\in H$, let $h_{T}=\pr_T(h)$ denote the projection of $h$ on $T$. Then
\begin{itemize}
  \item[$(a)$]$\gcdto(\vv,S(\D),D(\D))=\gcdto(\vv_{T})$.
  \item[$(b)$]If $y,v\in D\oplus T$, or if $y\in T$, $v\in H$, then $\ialg{v,y}=\ialg{v_T,y}$.
\end{itemize}
First possibility is $\gcdto(x,\ww',S(\D),D(\D))>1$. Consider the smallest summand $V$ in $T$ containing $\ideal{x,\ww'_T}$. By $(b)$, $V$ is isotropic. By $(a)$, $
\gcdto(x,\ww'_T)=\gcdto(x,\ww',S(\D),D(\D))>1
$, and likewise, $\gcdto(\ww'_T)=1$. This means there is a basis of $V$ of the form $\set{\tilde x, \ww'_T}$. By $(a)$ and $(b)$,
\begin{equation}\label{e:N}
\Ndelblankx{\ww'}(g)=\Ndelblank{\ww'; \ideal{\tilde x}^{\perp}}(g).
\end{equation}
One checks that $\tilde x$ satisfies all the requirements of the Proposition. So we can use $(i)$ by induction on the map in \eqref{e:induction} to obtain $\f$ as in \eqref{e:fgoal2} below.

The second possibility is $\gcdto(x,\ww',S(\D),D(\D))=0$. Actually,
\begin{equation*}
0=\gcdto(x,\ww',S(\D),D(\D))
=\gcd(x,\ww'_2,D(\D)),
\end{equation*}
where the second equality uses Remark \ref{r:gcdlig} along with $\ww'\in \Ndel(g)$. So consider the summand $V=\ideal{\ww_2',D(\D)}$; then $x\in V$. Further, since $\ww'\perp D(\D)$, we get $(\ww_2')^i-(\ww_T')^i\in D(\D)$ (here, $\vv^i$ denotes the $i$th vertex of $\vv$), and thus
\begin{equation*}
V=D(\D)+ \ideal{\ww_2'}=D(\D)\oplus \ideal{\ww'_T}.
\end{equation*}
So as a basis of $V$, we can take a basis of $D(\D)$ along with $x$ and vectors $t_0,\ldots,t_\ell$, such that $\set{x,t_0,\ldots,t_\ell}$ is a basis of $\ideal{\ww'_T}\del T$. We can choose them such that $\rk^x(t_j)=0$. Now for $v\in T$, set
$\bar v=a_1+tb_1+v+u$, where $u\in D(\D)$ is from Lemma \ref{r:compensate} such that $\ialg{\bar v,\D}=\ialg{a_1+tb_1+u,\D}=0$. We consider $(\bar x,\bar t_0,\ldots, \bar t_\ell)$, which is isotropic, since $V$ is easily shown to be isotropic. In fact it is a simplex in $\Ndel(g)$, since
\begin{eqnarray}\label{e:reuse}
    1&=&\gcdto(\ww',S(\D),D(\D))=\gcdto(\ww'_T,S(\D),D(\D))\\
&=&\gcdto(x,t_0,\ldots,t_\ell,S(\D),D(\D))= \gcdto(\bar x,\bar t_0,\ldots, \bar t_\ell,S(\D),D(\D))\nonumber
\end{eqnarray}
The last equality holds since $u\in D(\D)$.
A slight modification of \eqref{e:reuse} along with $(a)$ and $(b)$ shows
$\Ndelblankx{\ww'}(g)=\Ndelblankx{(\bar x,\bar t_0,\ldots, \bar t_\ell)}(g)$.
Completely analogously to the proof of Prop. \ref{p:conn1}, see \eqref{e:phi}, we can then factor  $f|_{\link(\s)}$ as
\begin{equation}\label{e:fgoal2}
    f|_{\link(\s)}:\link(\s)\del B\stackrel{\f}{\To} \Ndelblankx{\ww*f(\s)}(g).
\end{equation}
We now modify $\f$ to a map $\f'$ by performing division with remainder: Let $v=f(t)$ for some fixed vertex $t\in \s$, and write $v_s=\f(s)$ for $s\in B^{(0)}$. By division we obtain $q_s$ such that $\abs{\rk^x(v_s-q_sv)}<\abs{\rk^x(v)}=R$ for all $s\in B^{(0)}$. For $s\in\dd B=\link(\s)$ we take $q_s=0$ so we do not change $\f$ on $\link(\s)$. Let $u\in D(\D)$ be the vector from Lemma \ref{r:compensate} such that $a_1+tb_1+u\perp\D$. We then set, cf. Remark \ref{r:idea},
\begin{eqnarray*}
\f'(s)&=&v_s-q_sv+q_s(a_1+tb_1+u)\\ &=& a_1+tb_1+\pr_2(\f(s)-q_sf(t))+q_su,
\end{eqnarray*}
for $s\in B^{(0)}$. Then $\f'(s)$ is again in $\Ndelblankx{\ww*f(\s)}(g)$, as one checks by using $u\in D(\D)$, $D(\D)$ is isotropic and $x\in T$. Then $\rk^x(\f'(s))=\ialg{\f'(s),y}= \ialg{v_s-q_sv,y}$, since $y$ is the dual basis vector to $x$, and $y\in T$, so $y\perp u$. The result is thus a simplicial map $\f':B\To \Ndelblankx{\ww*f(\s)}(g)$ with $\abs{\rk^x(\f'(s))}<R$ for all $s\in B^{(0)}$. Then we do a link move to $f$ on $\s$ with $\f'$, which produces a map homotopic to $f$, removing $\s$. Iterating this, we obtain $R=0$.

We conclude $(ii)$ from $(i)$ precisely as in the proof of \ref{p:conn1} $(ii)$. 
\end{proof}

\begin{cor}$\Ldelgodeb(g)$ is $(g-k-3)$-connected.
\end{cor}
\begin{proof}Choose a dual summand $D=D(\D^k)$ to $S(\pr_2(\D^k))$. Then by Prop. \ref{p:putmansver}, $N_{\D^k}(g)=\Mdelblank{D}{0}{\D^k,\Ø,\Ø}(g)$ is $(g-k-3)$-connected. By Prop. \ref{p:reduce} $(iii)\Tto(i)$, this implies that $\Mdelblank{0}{0}{\Ø,\Ø,\D^k}(g)=\Ldelgodeb(g)$ is $(g-k-3)$-connected.
\end{proof}

\section{Exactness in the spectral sequence}\label{ss:exact}
Let $H(m)=H_1(F_{m,1};\Z)$ with given symplectic basis $(a_1,b_1,\ldots,a_m,b_m)$. Let $i=1,2$, put $H=H(g+i-1)$, and $H_{\Q}=H\tensor \Q$. Always assume $g\ge 6$.

In this section, we finish the proof of Theorem \ref{t:main}. Recall we must show
	\[E^2_{2,1}(F_{g,1};1)=0,\quad E^2_{2,1}(F_{g-1,2};2)=0.
\]
First we need a more concrete description of the spectral sequence for $q=1$. We shall use the following result of \cite{Berg}, Theorem 3.5.6. See also \cite{Putman3}, Theorem 1.2.

\begin{newthm}[\cite{Berg}]\label{t:Berg}
Let $S$ be a subsurface of $F_{g,1}$, obtained from $F_{g,1}$ by cutting along arcs $\g_1,\ldots,\g_n$, where $(\g_1,\ldots,\g_n)\in C_*(F_{g,1},1)$. Let $\tau_{F_{g,1},S}$ denote the restriction to $\I(F_{g,1},S)$ of the Johnson homomorphism $\tau_{g,1}: \I_{g,1}\To \Ltre H$.  Let $c_j$ be the homology class of $\tilde \g_j$ (see Figure \ref{f:closeup}) in $H_\Q$. Assume the genus of $S$ is at least $3$. Then
\begin{equation*}
    H_1(\I(F_{g,1},S);\Q)\iso \im(\tau_{F_{g,1},S})\tensor \Q \iso \Ltreort{c_1,\ldots,c_n}\del \LtreH.
\end{equation*}
\end{newthm}
\noindent Using this we get from \eqref{e:spectral2},
\begin{equation}\label{e:spectral2a}
    E^1_{p,1}(F_{g,1};1)\iso \bigoplus_{\ww\in\L(g)^{(p-1)}}\Ltreort{\ww}\Tto 0\text{ for } p+1\le g-1.
\end{equation}
Likewise from \eqref{e:spectral3}, using that $F_{g,2}=(F_{g+1,1})_\be$, where $\tilde \be\simeq\be_1$, see Figure \ref{f:beta},
\begin{equation}\label{e:spectral3a}
    E^1_{p,1}(F_{g,2};2)\iso \bigoplus_{\ww\in\Ll(g)^{(p-1)}}\Ltreort{\ww,b_1}\Tto 0\text{ for } p+1\le g.
\end{equation}

To enable us to talk about both cases simultaneously, define 
\begin{equation}\label{e:bi}
	\bi =\left\{
      \begin{array}{ll}
        \Ø,   & i=1, \\
   			b_1,	 & i=2.
      \end{array}
    \right.
\end{equation} 
The differentials $d^1_{p,1}$ have the following description under the isomorphisms \eqref{e:spectral2a} and \eqref{e:spectral3a} above: Let $\dd_j$ denote the $j$th face map in $\L(g)$, i.e. if $\ww=(w_0,\ldots,w_p)$ then $\dd_j\ww=(w_0,\ldots,\hat w_{j},\ldots,w_p)$. Write an element of $E^i_{p,1}$ as $(v,\ww)$ where $v\in \Ltreort{\ww,\bi}$. Then $d^1_{p,1}$ is the linear map given by
\begin{equation}\label{e:diff}
    d^1_{p,1}(v,\ww)= \sum_{j=0}^{p-1}(-1)^j(i_j(v),\dd_j\ww)
\end{equation}
where $i_j:\Ltreort{\ww,\bi}\To \Ltreort{\dd_j\ww,\bi}$ denotes the inclusion. 

\subsection{Morse vector field construction strategy}

From now on, write $E^1_{2,1}=E^1_{2,1}(F_{g,i};i)$. That $E^2_{2,1}=0$ is equivalent to the sequence  $E^1_{1,1}\longleftarrow E^1_{2,1}\longleftarrow E^1_{3,1}$ being exact. We shall use the technique of discrete Morse theory on the chain complex $\set{E^1_{n,1}}_{n\ge 0}$.

\begin{defn}\label{d:morse}Given a chain complex $C_*$ and bases $B_n=\set{b_n^j|j\in J_n}$ of each chain group $C_n$, then a vector field $V_*=\set{V_n}_{n\ge 0}$ on $\C_*$ is for each $n$ a collection of basis vector pairs, $V_n=\set{(b_n^j,b_{n+1}^j)|j\in J'_n}$ where $J_n'\del J_n$ is some subset, satisfying $(i)$ and $(ii)$:
\begin{itemize}
  \item[$(i)$]For each $j\in J'_n$, $d(b^j_{n+1})= b^j_n+\sum_{i\in J_n,i\neq j}c_ib_n^i$.
  \item[$(ii)$](Disjoint pairs): $\set{b_n^j,b_{n+1}^j}\cap \set{b_m^k,b_{m+1}^k}=\Ø$ if $(j,n)\neq (k,m)$.
\end{itemize}
A gradient path for $V_*$ is a sequence $(b_0, b_0')\to (b_1, b_1') \to (b_2,b_2') \to \cdots$, with $(b_j,b_j')\in V_*$, and for each $j$, $b_{j+1}$ has a nonzero coefficient in the basis expansion of $d(b'_{j})$, but $b_{j+1}\neq b_j$. If $V_*$ has no infinite gradient paths, we call it a \emph{Morse vector field}. 

Define $\mathcal R_n=\textup{span}\set{b_n^j|j\in J'_n}$, and $\mathcal C_{n+1}=\textup{span}\set{b_{n+1}^j|j\in J'_n}$, and write $c(b_j^n):=b_j^{n+1}$. For a subspace $A\del C_n$, we say $V_*$ spans $A$, if $A\del \mathcal R_n\oplus \mathcal C_{n}$. 
\end{defn}

We call vectors of $\mathcal R_n$ redundant, and vectors of $\mathcal C_n$ collapsible. The goal is to construct a Morse vector field that spans $E^1_{0,1}$, $E^1_{1,1}$, and $E^1_{2,1}$, i.e. for $n\le 2$, $E^1_{n,1}=\mathcal R_n\oplus \mathcal C_n$. Then it is easy to show that $E^1_{1,1}\longleftarrow E^1_{2,1}\longleftarrow E^1_{3,1}$ is exact.

Both $\L(g)$ and $\Ll(g)$ are multi-simplicial complexes. Given a \multi{} $\L$, and a total ordering $\Or$ on the vertices of $\L$, define $\Or \L$ to be the subcomplex of $\L$ consisting of simplices with vertices in ascending order. Then $\Or\L$ is a simplicial complex. We will first find vector fields on the simplicial complexes $\L^1(g)=\Or \L(g)$ and $\L^2(g)=\Or\Ll(g)$. We write a simplex in $\L^i(g)$ as a set of vectors $\set{v_0,\ldots,v_n}$, which is unambiguous since the order is fixed, to distinguish it from a simplex in $\L(g)$ or $\Ll(g)$.

\begin{defn}Let $i\in\set{1,2}$. Define $E^i_*$ to be the chain complex with chain groups $E^i_n=\bigoplus_{\ww\in \L^i(g)}\Ltreort{\ww,\bi}$, and differential as in \eqref{e:diff}. 
\end{defn}

\begin{rem}
The proofs from section \ref{ss:conn}, Prop. \ref{p:conn1a} and \ref{p:conn2}, work verbatim to give that $\L^1(g)$ and $\L^2(g)$ are $(g-2)$-connected.
\end{rem}

\begin{rem}\label{r:tactics}Given a choice of basis $B=B_n$ of $\mathcal R_n$, each basis vector $z\in B$ is in some $\Ltreort{\ww,\bi}$ for some simplex $\ww\in \L^i(g)^{(n-1)}$. We write $\ww=\simp(z)$. Then by $(i)$ in Def. \ref{d:morse}, we must have $\simp(c(z))=\vv\in \L^i(g)^{(n)}$, where $\ww$ is a face of $\vv$. The are natural inclusions $i_{\ww}:\Ltreort{\ww}\To \LtreH$ for each $\ww$. Then $c(z)\in E^i_{n+1}$ is specified uniquely by requiring that $i_{\ww}(z)=i_{\vv}(c(z))$. So to define the vector field $V_n$, we need only specify the basis $B_n$ of $\mathcal R_n$, and for each basis vector $z\in B_n$ choose $\simp(c(z))\in \L^i(g)$. 
\end{rem}
We need dual vectors, as in Prop. \ref{p:dual}. We choose fixed dual vectors to the vertices $A_1,A_2,A_3, A_4$ as follows:
\begin{equation}\label{e:startdual}
B_1=\left\{
      \begin{array}{ll}
        b_1, & \hbox{for }\L^1(g) \\
        b_1-b_2-b_3-b_4, & \hbox{for }\L^2(g);
      \end{array}
    \right.
\quad B_j=b_j\text{ for } j=2,3,4.
\end{equation}

\begin{lem}\label{l:decompose}
Let $i\in\set{1,2}$. Given a simplex $\ww\in \L^i(g)$, assume there is $\set{w_1,w_2,w_3,w_4}\in \L^i(g)$ with $\ww * \set{w_1,w_2,w_3,w_4}\in\L^i(g)$. Let $u_j$ be a dual vector to $w_j$ with $\ialg{w,u_j}=0$ for each vertex $w\in\ww$ and each $j=1,2,3,4$. Then there is an isomorphism
\begin{eqnarray*}
    \qquad \qquad\Ltreort{\ww,\bi}&\iso&\Ltreort{\ww,w_1,\bi}\oplus \ideal{u_1}\wedge \Lort{2}{\ww,w_2,\bi}\\
    &&\oplus \ideal{u_1,u_2}\wedge\ortQ{\ww,w_3,\bi}\oplus\ideal{u_1,u_2,u_3}.\qquad \qquad\qed
\end{eqnarray*}
\end{lem}
We call this the decomposition of $\Ltreort{\ww,\bi}$ with respect to $w_1,w_2,w_3,w_4$. (Note it also depends on a choice of dual vectors.)

Using Lemma \ref{l:decompose}, we can define a vector field on $\Ltreort{\ww,\bi}$ as follows:
\begin{cor}\label{c:morse}Given $\ww\in \L^i(g)$. Assume there are $w_1,w_2,w_3,w_4$ with $\ww*\set{w_1,w_2,w_3,w_4}\in \L^i(g)$. Then there is a vector field that spans $\Ltreort{\ww,\bi}$.
\end{cor}
\begin{proof}Write $\ww=\set{v_1,\ldots,v_n}$. Choose dual vectors $y_1,\ldots,y_n,u_1,\ldots,u_4$ to $v_1,\ldots,v_n,w_1,\ldots,w_4$; then we can use Lemma \ref{l:decompose} to decompose $\Ltreort{\ww,\bi}$. Choose a basis $B_{w_i}$ of each of the four summands in the decomposition. For $z\in B_{w_i}$, set $\simp(c(z))=\ww*\set{w_i}$. Then set $B_w=B_{w_1}\cup B_{w_2}\cup B_{w_3}\cup B_{w_4}$.
\end{proof}

After having done as in Cor. \ref{c:morse} above, we have defined some collapsible vectors inside $\bigoplus_{j=1}^4\Ltreort{\ww,w_j}$. We will need to make the rest of these four summands redundant:

\begin{lem}\label{l:decomposerest}
Given $\ww\in \L^i(g)$. Assume there are $w_1,w_2,w_3,w_4$ with $\ww*\set{w_1,w_2,w_3,w_4}\in \L^i(g)$. Let $C(w_j)$ be the summands of Lemma \ref{l:decompose},
\begin{eqnarray*}
&&\mathcal C(w_1)=\Ltreort{\ww,w_1,\bi},\qquad\qquad\quad\!\mathcal C(w_2)=\ideal{u_1}\wedge\Lort{2}{\ww,w_2,\bi},\\
&&\mathcal C(w_3)=\ideal{u_1\wedge u_2}\wedge\ortQ{\ww,w_3,\bi},\quad
\mathcal C(w_4)=\ideal{u_1\wedge u_2\wedge u_3}.
\end{eqnarray*}
Define: $\mathcal R(w_1)=0$, $\mathcal R(w_2)=\Ltreort{\ww,w_1,w_2,\bi}$, and
\begin{eqnarray*}
\mathcal R(w_3)&=& \Ltreort{\ww,w_1,w_3,\bi}\oplus\ideal{u_1}\wedge\Lort{2}{\ww,w_2,w_3,\bi},
\\
\mathcal R(w_4)&=&\Ltreort{\ww,w_1,w_4,\bi}\oplus \ideal{u_1}\wedge\Lort{2}{\ww,w_2,w_4,\bi}\\
								&&\oplus \ideal{u_1\wedge u_2}\wedge\ortQ{\ww,w_3,w_4,\bi}.
\end{eqnarray*}
Then, $\Ltreort{\ww,w_j}=\mathcal C(w_j)\oplus \mathcal R(w_j)$.\qed
\end{lem}
\begin{cor}\label{c:morserest}Given $\ww\in \L^i(g)$. Assume there are $w_1,w_2,w_3,w_4$ with $\ww*\set{w_1,w_2,w_3,w_4}\in \L^i(g)$. Then there is a vector field on $E^i_*$ that spans 
$\bigoplus_{j=1}^4\Ltreort{\ww,w_j,\bi}
$.
\end{cor}
\begin{proof}
From Cor. \ref{c:morse} we have chosen $\bigoplus_{j=1}^4 \mathcal C(w_j)$ to be collapsible inside $\bigoplus_{j=1}^4\Ltreort{\ww,w_j}$, in the notation of Lemma \ref{l:decomposerest}. This Lemma also gives a decomposition of the rest, allowing us to choose bases $B(w_j)$ for $R(w_j)$ of the form $B(w_j)=B_1(w_j)\cup \cdots \cup B_{j-1}(w_j)$. Here $B_1(w_j)$ is a basis of $\Ltreort{\ww,w_1,w_j}$, $B_2(w_j)$ is a basis of $\ideal{u_1}\wedge \Lort{2}{\ww,w_2,w_j}$, and $B_3(w_4)$ is a basis of $\ideal{u_1\wedge u_2}\wedge\ortQ{\ww,w_3,w_4}$. For each basis vector $b\in B_m(w_j)$ we assign $\simp(c(b))=\ww*\set{w_m,w_j}$.
\end{proof}

The difficulty will be ensuring there are no infinite gradient paths for the constructed vector field. The idea is define filtrations 
\begin{itemize}
  \item $\F^1_\infty\subset \F^1_5 \subset \F^1_4 \subset \F^1_3 \subset \F^1_2 \subset \F^1_1=\L^1(g)$ on $\L^1(g)$.
  \item $\F^2_\infty\subset\F^2_5\subset\F^2_4\subset \F^2_3 \subset \F^2_2 \subset\F^2_1\subset\F^2_0=\L^2(g)$ on $\L^2(g)$.
\end{itemize}
Then for each simplex $\ww$, we will pick such $w_1,w_2,w_3,w_4$ in a smaller filtration than $\ww$, allowing us to argue inductively. We introduce the notation:
\begin{equation}\label{e:start}
    A_1=a_1,\quad A_j=\left\{
                        \begin{array}{ll}
                          a_j, & \hbox{in }\L^1(g) \\
                          a_1+a_j, & \hbox{in }\L^2(g);
                        \end{array}
                      \right. \text{ for }j=2,3,4.
\end{equation}
\begin{defn}
Define $\F_\infty^i$ to be the full subcomplex of $\L^i(g)$ with vertices in  $\set{A_1,A_2,A_3,A_4}$. Define $\F^1_k$ to be the full subcomplex of $\L^1(g)$ with vertices in  $H_k\cup\set{A_1,A_2,A_3,A_4}$, where $H_k=\ideal{a_k,b_k,\ldots, a_g,b_g}\del H$.

To define $\F^2_k$, let $\tilde \F^2_k$ be the full subcomplex of $\L^2(g)$ with vertices in $(a_1+H_k)\cup\set{A_1,A_2,A_3,A_4}$, and set $\F^2_k=\star_{\L^2(g)}(a_1)\cap \tilde\F^2_k$ for $k\ge 2$. For $k=1$ we set $\F_1^2=\Or\Ll_{\gcd=1}(g)\cup \F^2_2$, see Def. \ref{d:gode}.

As a general convention $\F^i_{5+1}=\F^i_\infty$. (We do not call it $\F^i_6$ since it does not agree with the definition of $\F^i_k$ for $k=6$.) 
\end{defn}

\begin{lem}\label{l:1234} Let $i\in\set{1,2}$ and $k\le 5$. To each $s$-simplex $\ww\in \F^i_k\fra \F^i_{k+1}$, we can assign a simplex $\ww^*=\set{w_1,\ldots, w_{g-3+i-s}}\in \F^i_{k+1}$ such that 
\begin{itemize}
	\item[$(i)$] $\ww* \ww^*$ is a simplex in $\F_k^i(g)$,
	\item[$(ii)$] $w_j=A_j$ for $1\le j\le k-1$,
	\item[$(iii)$] If $\vv\del \ww$ with $\vv\in \F_{k+1}^i(g)$, then $\vv*\ww^*\in\F_{k+1}^i(g)$,
	\item[$(iv)$]For $i=2$ and $k=0$, write $\ww=\set{\ww_0,\ldots,\ww_s}\in\L^2(g)$, then $\rk^{b_1}(\ww_0)=\rk^{b_1}(w_j)$ for all $j$.
\end{itemize}
\end{lem}
\begin{proof}
We set $w_j=A_j$ for $1\le j\le k-1$, and now define the remaining $w_j$ for $k\le j\le g-3+i-s$. Set $m=\max(k+1,2)$, and let $S=S(\pr_{m}(\ww))$. Then use Prop. \ref{p:dual} to obtain a dual summand $D$ in $H_{m}$, and a symplectic subspace $T$ such that $(S\oplus D)\oplus T = H_{m}$. We want to choose $w_j$ using the basis vectors $\set{t_1,\ldots,t_n}$ for an isotropic summand of $T$. A dimension count shows $\dim T\ge 2(g-1+i-s-m)$. For $i=1$, we put $w_j=t_j$, and are done. 

Now let $i=2$. We put $w_j=a_1+t_j$ when $k\ge 2$. When $k\le 1$ this may not work since $a_1+t_j$ may not be orthogonal to $\ww$. But by Lemma \ref{r:compensate} we can choose $u\in D$ such that $a_1+tb_1+u\perp\ww$ for $t\in \Z$ as in \eqref{e:t}. Then set $w_j=a_1+tb_1+u+t_j$. $(iii)$ is automatic except for $k=0$, where it follows from Remark \ref{r:gcdlig}, using $\gcdto(\vv)=1$ and the vectors of $\pr_2(\vv)$ are in $S$, while $w_j\in D\oplus T$. $(iv)$ follows from the choice of $t$.
\end{proof}
Note assuming $g\ge 8-i$, we always get (at least) a $4$-simplex $x^*=\set{x_1,x_2,x_3,x_4,x_5}\in \F^i_{k+1}$ assigned to each vertex $x\in\F^i_{k}$. We emphasize that for each vertex $x$ we choose $x^*$ once and for all. 

\subsection{Implementing the general strategy}
\begin{thm}\label{t:vectfield2}Assume $g\ge 8-i$, and let $i\in\set{1,2}$. There is a Morse vector field on $E^i_*$ that spans $E^i_0, E^i_1$, and $E^i_2$.
\end{thm}

Proving this will will the aim of this section. We first show how to extend Theorem \ref{t:vectfield2} to $\L(g)$ and $\Ll(g)$:

\begin{prop}\label{p:order} Let $i\in\set{1,2}$ and $g\ge 8-i$. Assume that we have a Morse vector field on $E^i_*$ spanning $E^i_0, E^i_1$, and $E^i_2$. 
Then there exists a Morse vector field on $E^1_{*,1}$ spanning $E_{0,1}^1$, $E_{1,1}^1$ and $E_{2,1}^1$.
\end{prop}
\begin{proof}
For a simplex $\ww\in\L(g)$, we write $\Or(\ww)\in \L^i(g)$ for the simplex with the same vertices placed in ascending order and $\streg\Or(\ww)$ for the descending order. We proceed by induction in the filtration degree $k$ to show that given a simplex $(x,y)=\streg\Or(x,y)$ with $(y,x)\in\F^i_k$ we can extend the vector field by a basis for the summand $\Ltreort{x,y,\bi}$ indexed by $(x,y)$.

Let $(y,x)\in \F^i_{k}(g)$, and assume one of $x$ and $y$, say $y$, is in $\F^i_{k+1}(g)$. By Lemma \ref{l:1234}, we get a 4-simplex $(y,x)^*=\Or(w_1,w_2,w_3,w_4)\in \F^{i}_{k+1} $ and by Lemma \ref{l:decompose} we get the basis $B(x,y)$ for $\Ltreort{x,y,\bi}$ such that for each $z\in B(x,y)$ there is $w_z=w_j$ for a $j\in\set{1,2,3,4}$ with $z\in \Ltreort{x,y,w_{z},\bi}$. Then let $\simp(c(z))=(\Or(x,w_z),y)$. The other two differentials from $(\Or(x,w_z),y)$ land in summands indexed by, respectively, $\Or(x,w_z)\in E^i_{2}$, and $(w_z,y)$ with $\Or(w_z,y)\in\F^i_{k+1}$ by Lemma \ref{l:1234} $(iii)$. So we create no infinite gradient paths, in the first case by the assumption of the lemma, and in the second case by induction in $k$.

We must show that the vectors $c(z)$ thus defined are linearly independent. 
The only 1-simplices besides $(x,y)$ which could get paired with $(\Or(x,w_z),y)$ in the vector field are ${\Or}(x,w_z)$ and $(w_z,y)$. The former is in $\L^i(g)$ so its partner is likewise in $\L^i(g)$ by assumption, and by inspection, it is easy to see that the latter cannot get paired with $(\Or(x,w_z),y)$ by the method above.

To finish the induction step, assume $x,y\in \F^i_k\fra \F^i_{k+1}$, and proceed exactly as above. The simplices ${\Or}(x,w_z)$ and $(w_z,y)$ now have one vertex in $\F^i_{k+1}$, and so are dealt with above. The induction start deals with $x,y$ among $A_1,\ldots,A_4$ by a similar method, using as $\set{w_1,\ldots,w_4}$ the four remaining among $A_1,\ldots,A_6$. 
\end{proof}
As a corollary, for $g\ge 7$ we have $E^2_{2,1}(F_{g,1};1)=0$ and $E^2_{2,1}(F_{g-1,2};2)=0$, which finishes the proof of the Main Theorem \ref{t:main}.

So we must prove Theorem \ref{t:vectfield2}. The easy part is degree 0 and 1:
\begin{prop}\label{l:C2}Let $i\in\set{1,2}$. There is a Morse vector field on $E^i_*$ that spans $E^i_0, E^i_1$, and where $\mathcal C_2$ is given by $\mathcal C_2=\mathcal C_2(\infty)\oplus \mathcal C'_2$. 

Here, $\mathcal C_2(\infty)= \mathcal R(A_2)\oplus  \mathcal R(A_3)\oplus  \mathcal R(A_4)$ in the notation of Lemma \ref{l:decomposerest} with $\ww=\bi$.
And letting $x^*=\set{x_1,x_2,x_3,x_4}$ as in Lemma \ref{l:1234}, 
\begin{equation*}
\mathcal C'_2=\hspace{-0.6cm}\bigoplus_{\substack{x\in \L^i(g)^{(0)}\\ x\notin\set{A_1,\ldots,A_4}}}\hspace{-0.6cm}\Ltreort{x,x_1,\bi}\oplus \ideal{y_1}\wedge\Lort{2}{x,x_2,\bi} \oplus \ideal{y_1\wedge y_2} \wedge \ortQ{x,x_3,\bi}\oplus \ideal{y_1\wedge y_2\wedge y_3}
\end{equation*}
where $y_j$ denotes a choice of dual vector to $x_j$.
\end{prop}
\begin{proof}
In degree $0$, we have $E^i_{0}=\Ltreort{\bi}$. We use Cor. \ref{c:morse} on $\ww=\bi$ with $w_j=A_j$ and $u_j=B_j$ for $j=1,2,3,4$, see \eqref{e:start} and \eqref{e:startdual}. This yields a vector field that spans $E^i_0$, and no gradient paths at all, since $d=\dd_0$.

Now consider degree 1. $\mathcal C_1$ can be read off from Lemma \ref{l:decompose} (here, $i=1$):
\begin{eqnarray*}
 \mathcal C_1&=& \Ltreort{A_1}\oplus \ideal{B_1}\wedge\Lort{2}{A_2}\oplus \ideal{B_1 \wedge B_2}\wedge \ortQ{A_3}\oplus \ideal{B_1 \wedge B_2 \wedge B_3} \nonumber\\
    &\del&  \Ltreort{A_1}\oplus\Ltreort{A_2}\qquad\;\;\;\:\oplus  \Ltreort{A_3}\qquad\quad\;\;\;\;\,\oplus \Ltreort{A_4}
\end{eqnarray*}
Consequently, $\mathcal R_1$ has to be the rest of $E^i_{1}=\bigoplus_{x\in \L^i(g)^{(0)}}\Ltreort{x,\bi}$.

For $A_1,A_2,A_3,A_4$, we use Cor. \ref{c:morserest}. In general, for each $x\in \L^i(g)^{(0)}$, $x\neq A_1,A_2,A_3,A_4$, we choose the vector field on $\Ltreort{x,\bi}$ as in Cor. \ref{c:morse}, using the four vectors of the simplex $x^*=\set{x_1,x_2,x_3,x_4}$ provided by Lemma \ref{l:1234}. Then we get a vector field spanning $E^i_1$ with $\mathcal C_2$ as stated.

We now argue why there are no infinite gradient paths starting in $\mathcal R_1$. By construction, if there is $(b,b')\to (d,d')$ where $\simp(b)=x$, then $\simp(d)=x_j$ for some $j=1,2,3,4$. We have ensured that if $x\in \F^i_k$ then $x_j\in \F^i_{k+1}$, and if $k=5$, then $x_j=A_j$. Thus any gradient path reaches some $A_j$ ($j=1,2,3,4$) in at most five steps. And if $\simp(b)=A_j$, by construction either $d\in \mathcal C_1$, or $\simp(d)=A_k$ with $k<j$. We see any gradient path has length $\le 9$.
\end{proof}

Now we will extend the vector field defined in Prop. \ref{l:C2} to degree $2$ by induction in the filtration $\F^i_k$. The induction assumption is $A(k)$:
\begin{description}
  \item[$A(k)$:] There is a Morse vector field $V^i_*$ on $E^i_*$ that spans  $E^i_0, E^i_1$, and $E^i_{2}(k):= \bigoplus_{\ww\in \F^i_k(g)^{(1)}}\Ltreort{\ww,\bi}$, for $i=1,2$.
\end{description}

Using Cor. \ref{c:morserest}, it is easy to prove the induction start $A(\infty)$. We now inductively assume $A(k+1)$, $k\in \set{1,2,3,4,5}$ (interpreting $5+1$ as $\infty$), and we wish to prove $A(k)$. This will be done in the next lemmas.

\begin{lem}\label{l:i-start}We can extend the Morse vector field $V^i_2$ so it spans
\begin{equation}\label{e:Ekstart}
    \bigoplus_{x\in \F_k^i(g)^{(0)}}\bigoplus_{j=1}^4\Ltreort{x, x_j,\bi}\del E^i_2(k).
\end{equation}
\end{lem}
\begin{proof}Note that \eqref{e:Ekstart} is precisely the part of $E^i_2(k)$ which intersects $\mathcal C_2$ non-trivially. We know $\mathcal C_2$ from Prop. \ref{l:C2}, $\mathcal C_2= \mathcal C'_2\oplus \mathcal C_2(\infty)$. The vector field on the summands of $E^i_{2}(k)$ intersecting $\mathcal C_2(\infty)$ is the induction start, so we must extend it to the summands intersecting $\mathcal C'_2$. Apply Cor. \ref{c:morserest} for each $x\in\F_k^i(g)^{(0)}\fra \set{A_1,\ldots,A_4}$ to construct the vector field. There are no infinite gradient paths by induction in $k$, precisely as in the proof of Prop. \ref{l:C2}.
\end{proof}

For now, take a fixed vertex $x\in \F^i_k\fra \F^i_{k+1}$. Define $E^i(x)\del E^i_2(k)$ by
\begin{equation}\label{e:Ex}
    E^i(x)=\bigoplus_{x'\in \F^i_{k+1}(g)^{(0)}:\set{x,x'}\in \F^{i}_{k}(g)}\Ltreort{x,x',\bi}.
\end{equation}
We will choose a basis of $E^i(x)$ in the next lemmas. Our method does not choose a basis of each summand $\Ltreort{x,x',\bi}$ at a time, instead mixing them up; this is so that we can argue there are no infinite gradient paths.

\begin{lem}\label{l:Ax}For a fixed vertex $x\in \F^i_k\fra \F^i_{k+1}$, there is a Morse vector field extending $V^i_*$ that spans a subset $A^i_x\del E^i(x)$ ($A^i_x$ is defined in the proof).
\end{lem}
\begin{proof}We can write $E^i(x)=\bigoplus_{x'\in \F^i_{k+1}(g)^{(0)}}\ideal{z\mid (z,x')\in S}$, where $S$ is the following index set:
\begin{equation*}
    S=\set{(z,x')\in \LtreH \times \F^i_{k+1}(g)^{(0)}\mid \set{x,x'} \in \F^i_{k+1}(g)^{(1)}, z\in \Ltreort{x,x',\bi}}
\end{equation*}
We filter $E^i(x)$ by subsets $F_0\del F_1\del F_2\del \cdots \del E^1(x)$, by inductively defining $S_0\del S_1\del\cdots \del S$, and setting $F_j=\bigoplus_{x'\in \F^i_{k+1}}\ideal{z\mid (z,x')\in S_j}$. For the induction start, define $S_0= \set{(z,x')\in S\mid x'\in \set{x_1,x_2,x_3,x_4}}$, and note $F_0$ is already handled in the proof of Lemma \ref{l:i-start}. For the induction step, assume we have defined $F_{j-1}$ and its basis $B_{j-1}$. Then set
\begin{eqnarray}\label{e:Sj}
    S_j=\left\{(z,x')\in S \right.&\mid &\exists x'':(z,x'')\in F_{j-1},\nonumber \set{x,x',x''}\in \F^i_{k+1}(g)^{(2)},\\
    										&& \left. z\in \Ltreort{x,x',x'',\bi}\right\}.
\end{eqnarray}
Choose a basis $B_j$ for $F_j$ by extending $B_{j-1}$ by suitable vectors $z$ with $(z,x')\in S_j$. Set $A^i_x=\bigcup_{i=1}^{\infty}F_j\del E^i(x)$. We get the basis $B_x= \bigcup_{j=1}^{\infty}B_j$ for $A^i_x$. To define the vector field, for $z\in B_j$, set $\simp(c(z)))=\set{x,x',x''}$ with $x''$ as in \eqref{e:Sj}. The vector field spans $A^i_x$ by construction.

We prove that there are no infinite gradient paths by induction in $j$. The induction start is Lemma \ref{l:i-start}. A gradient path from $z\in B_j$ leads to either a vector with simplex $\set{x',x''}\in \F^i_{k+1}$ or to $F_{j-1}$; and there are no infinite gradient paths starting there by induction in $k$ or $j$, respectively.
\end{proof}

\begin{prop}\label{p:A=E}If $g\ge 8-i$, $A^i_x=E^i(x)$.
\end{prop}
This is a major step, and the differences between $i=1,2$ become so pronounced that we handle them separately, in the next two subsections. First, though, we observe it suffices to prove the following:

\begin{lem}\label{l:A=E}Let $x'\in \F^1_{k+1}$ such that $\set{x,x'}$ is a simplex be given. Then there exist a basis $\overline{B}(x')$ for $\Ltreort{x,x'}$ such that for all $z\in \overline{B}(x')$ there is $s\in \N$ with $(z,x')\in F_s$. 
\end{lem}

\begin{rem}\label{r:start} For $x_1,\ldots,x_5$ associated to $x$ from Lemma \ref{l:1234}, there are bases $\streg B(x_m)$ satisfying Lemma \ref{l:A=E}: Given any $z\in \Ltreort{x,x_m,b_1}$, from the definition it is easy to see that $(x_m,z)\in F_1$, where $m=1,2,3,4,5$.
\end{rem}

\begin{lem}\label{l:Ak}
Assuming Lemma \ref{l:A=E}, $A(k)$ holds.
\end{lem}
\begin{proof}
We have then by Lemma \ref{l:Ax} a Morse vector field spanning
\begin{equation*}
E^i(k,k+1):=\bigoplus_{x\in \F^i_k\fra \F^i_{k+1}}E^i(x)=\hspace{-1cm}\bigoplus_{\set{x,x'}\in\F^i_k(g)^{(1)}: x\in \F^i_k\fra\F^i_{k+1},x'\in\F^i_{k+1}}\hspace{-1cm}\Ltreort{x,x',\bi},
\end{equation*}
All we are missing are the cases $x,x'\in \F^i_k\fra \F^i_{k+1}$ and $x,x'\in \F^i_{k+1}$ but $\set{x,x'}\in\F^i_k\fra\F^i_{k+1}$, which we handle in the same way: We take $\set{x,x'}^*=\set{v_1,v_2,v_3,v_4}$ from Lemma \ref{l:1234}. Construct the vector field as in Cor. \ref{c:morserest}. Given a gradient path starting in some vector with simplex $\set{x,x'}$, it can either lead to $\set{x,v_i}$ or $\set{x',v_i}$, both of which are simplices for vectors in $E^i(k,k+1)$ in the first case, or in $E_2^i(k+1)$ in the second case by Lemma \ref{l:1234} $(iii)$, and no infinite gradient paths start there, in the latter case by $A(k+1)$.
\end{proof}

Pending the proof of Lemma \ref{l:A=E}, we have now proved Theorem \ref{t:vectfield2}.

\subsection{Finishing the proof of Theorem \ref{t:vectfield2} for $i=1$}

In this section, $i=1$. We must prove Lemma \ref{l:A=E}. In light of Remark \ref{r:start}, we will deduce the existence of $\streg B(x')$ from that of $\streg B(x_m)$, where $m=1,2,3,4,5$. To do this, we need some preliminary Lemmas.

\begin{lem}\label{l:basis}
Given $\set{x,x'}\in\F^1_k(g)$ with $x,x'\notin \set{a_1\ldots, a_{k-1}}$. Assume given $v_1,v_2,v_3,v_4\in \F_{k+1}^1(g)^{(0)}$ such that $\set{x,v_1,v_2,v_3,v_4}\in \F_k^1(g)$, and $\gcd(x,x',v_1,v_2,v_3,v_4)>0$, and $v_j=a_j$ for $1\le j\le k-1$. Then,
\begin{itemize}
  \item[$(i)$]There is a $\Q$-basis $B=B(x')$ for $\ortZ{x,x'}$ that satisfies:
      \begin{itemize}
        \item[$(a)$]$\set{a_1,b_1,\ldots,a_{k-1},b_{k-1}}\del B\del \F^1_k(g)^{(0)}\cup \set{b_1,\ldots, b_{k-1}}$,
        \item[$(b)$]$\forall z\in B: \abs{\set{1\le j\le 4\mid z\perp v_j}}\ge 3$.
      \end{itemize}
  \item[$(ii)$]There is a basis $\overline B=\overline B(x')$ for $\Ltreort{x,x'}$ such that for each $\bar z\in \overline B$ there is $ 1\le j\le 4$ with $\bar z\in\Ltreort{x,v_j}$.
\end{itemize}
\end{lem}
\begin{proof}
Write $H_{<k}= \ideal{a_1,b_1,\ldots,a_{k-1},b_{k-1}}$. By assumption, $x,x'\in H_{k}$ so $\ortZ{x,x'}=H_{<k}\oplus V_k$, where $V_k\del H_k$. We choose $B=\set{a_1,b_1,\ldots,a_{k-1},b_{k-1}}\cup B_k$, where $B_k\del \F^1_k(g)^{(0)}$ is a basis of $V_k$ that satisfies $(b)$.

To construct $B_k$, consider $\ortZ{x,x',v_k,\ldots,v_4}$. Since $v_i\in H_k$ for $i\ge k$, we have again $\ortZ{x,x',v_k,\ldots,v_4}=H_{<k}\oplus V_k'$, where $V_k'\del V_k$. Take a basis $B_k'$ of $V_k'$.
Now we choose a type of ``dual vectors'' to $v_k,\ldots,v_4$, as follows: Consider, for $k\le j\le 4$, the space $\ortZ{x,x',v_k,\ldots,\hat v_j,\ldots,v_4}=H_{<k}\oplus V_k'(j)
$,
where $V_k'\del V_k'(j)$. Choose $u_j\in V_k'(j)$ such that it maps to a simple vector under the quotient map $V_k'(j)\To V_k'(j)/V_k'$. Then $u_j\in \F^1_k(g)$, and by construction, $B_k:=B_k'\cup \set{u_k,\ldots,u_4}$ is a $\Q$-basis of $V_k$ that satisfies $(b)$. This shows $(i)$, and $(ii)$ follows immediately.
\end{proof}

The following gives us a way to ensure that a basis element $z_1\wedge z_2 \wedge z_3\in \Ltreort{x,x_j}$ can be ``carried along'' to $\Ltreort{x,x'}$. We owe the idea to A. Putman, from his master class \emph{The Torelli group} at Aarhus University, 2008.
\begin{defn}
Let $i\in\set{1,2}$, and let $0\le k\le 5$. Let $x\in \F^i_k\fra \F^i_{k+1}$. For $z_1,\ldots,z_n\in \ortQ{x}$, set $Z=\ortQ{z_1,\ldots,z_n}$, and let $W^i$ denote the full subcomplex of $\L^i(g)$ with vertices in $Z$. Define the complex $\Lconn{i}(g)$ to be 
\begin{equation*}
  \Lconn{i}(g)=\set{\vv\in \F^i_{k+1}\cap W^i\mid \set{x,\vv}\in \F^i_{k}}.
\end{equation*}
\end{defn}
\noindent Note, for a simplex $\vv\in \Lconn{1}(g)$, if $n=3$, we have $z_1\wedge z_2 \wedge z_3\in \Ltreort{x,\vv}$.

\begin{lem}\label{l:smh}
Let $n+4\le g$. Given a basis $B$ of $\ortQ{x,x'}$ that satisfies $(a)$ of Lemma \ref{l:basis} $(i)$. Assume $z_1,\ldots,z_n\in B$. Then $\Lconn{1}(g)$ is connected.
\end{lem}
\begin{proof}
First,  assume $a_j\in Z$ for some $1\le j <k\le 5$. Then $a_j\in \Lconn{1}(g)$, and moreover for any $v\in \Lconn{1}(g)^{(0)}\fra \set{a_j}$ we have $\set{a_j,v}\in \Lconn{1}(g)^{(1)}$. Thus $\Lconn{1}(g)$ is connected in this case.
Note this takes care of $n<k-1$.

Now let $n\ge k-1$. We can assume $a_j\notin Z$ for all $j<k$. Since $z_m\in B$, we must have $z_j=b_j$ for $1\le j\le k-1$. Then $\F^1_{k+1}(g)\cap Z=\F^1_{k+1}(g)\cap \ortQ{z_k,\ldots,z_n}\fra \set{a_1,\ldots,a_{k-1}}$. So if we set $Z'=\ortQ{z_k,\ldots,z_n}\fra\set{a_1,\ldots,a_{k-1}}$, then $\Lconn{1}(g)=\LconnZ{Z'}(g)$. Thus it suffices show $\LconnZ{Z'}(g)$ is connected.

For now, assume $n\ge k$. Given $v_1,v_2\in \LconnZ{Z'}(g)^{(0)}$, we will show there is a path $v_1\to w\to v_2$ in $\LconnZ{Z'}(g)$. Let  $V^1=\ideal{v_1,v_2,x}$ and $V^2=\ideal{z_k,\ldots,z_n}$. The idea is that $w$ need not be independent of $z_k\ldots,z_n$, only orthogonal to them. Let  $SV^j_{k+1}$ be the smallest summand in $H_{k+1}$ containing $V^j_{k+1}=\pr_{k+1}(V^j)$. Take a dual summand $D^1$ (cf Prop. \ref{p:dual}) to $SV^1_{k+1}$ in $H_{k+1}$, i.e. there is $T^1$ so that $(SV^1_{k+1}\oplus D^1)\oplus T^1=H_{k+1}$. Now take a dual $D^2$ to $SV^{2,1}_{k+1}=S(\pr_{T_1}(V^2_{k+1}))$ in $T^1$, so obtaining $T^2\perp V^1+V^2$ with
\begin{equation}\label{e:duals}
    (SV^1_{k+1}\oplus D^1)\oplus (SV^{2,1}_{k+1}\oplus D^2)\oplus T^2=H_{k+1}.
\end{equation}
Since $n\ge k$, if $SV^{2,1}_{k+1}=0$, then $\dim T^2\ge 2(g-k-3)\ge 2$. Thus, any simple vector $w\in T^2$ gives a path $v_1\to w\to v_2$ in $\LconnZ{Z'}(g)$. If $SV^{2,1}_{k+1}\neq 0$, then choose $w'\in SV^{2,1}_{k+1}$, and use the dual basis for $D^2$ to modify $w'\rightsquigarrow w$ such that $w\perp z_i$, $i=k,\ldots,3$. Then $w\in (SV^{2,1}_{k+1}\oplus D^2)$, and so forms a simplex with $x,v_1,v_2$, since they are in $SV^1_{k+1}$. This proves the Lemma for $n \ge k$.

Now let $n=k-1\le 4$, i.e. $Z'=H\fra\set{a_1,\ldots,a_{k-1}}$. Again given $v_1,v_2\in \LconnZ{Z'}(g)^{(0)}$, set $\streg x=\pr_{H_{k+1}}(x)/\gcd(\pr_{H_{k+1}}(x))$. Consider $\bar V^j=\ideal{v_j,\streg x}$ for $j=1,2$, project to $H_{k+1}$ and take the dual as above, $(S\bar V^j_{k+1}\oplus D^j)\oplus \bar T^j_1=H_{k+1}$, and the dimension argument above gives $\bar T^j_1\neq 0$. We obtain a simple vector $w_j\in \bar T_1^j$. Define the subcomplex $\bar\F_{k+1}(g)\del \F^1_{k+1}(g)$ of simplices with vertices in $H_{k+1}$. Thus, $w_j\in\link_{\bar\F_{k+1}(g)}(\bar x)$. Now, $\bar\F_{k+1}(g)\iso \L(g-k)$ via $H_{k+1}\iso H(g-k)$. Then by Prop. \ref{p:conn1}, $\link_{\bar\F_{k+1}(g)}(\bar x)$ is connected for $g\ge k+3=n+4$, so there is a path from $w_1$ to $w_2$. This yields the desired path 
in $\LconnZ{Z'}(g)$, namely $v_1\to w_1\to \cdots \to w_2\to v_2$.
\end{proof}

\begin{proof}[Proof of Lemma \ref{l:A=E} for $i=1$:]
As usual $x_1,\ldots,x_5$ is the vectors associated to $x$ from Lemma \ref{l:1234}. Let $\overline{B}(x')$ be the basis for $\Ltreort{x,x'}$ from Lemma $\ref{l:basis}$, where $v_1,\ldots,v_4$ are four among $x_1,\ldots,x_5$. So for each $z=z_1\wedge z_2\wedge z_3\in \overline{B}$ there is $1\le j\le5$ with $z\in \Ltreort{x,x_j}$. For each $z\in \streg B$, set $Z=\ideal{z_1,z_2,z_3}^{\perp}$, and consider the complex $\Lconn{1}(g)$ from Lemma \ref{l:smh}, which is connected. We see $x',x_j\in \Lconn{1}(g)$, so there is a path in $\Lconn{1}(g)$ connecting them, $x_j=v_0\to v_1\to\cdots\to v_\ell=x'$. Then $(z,x')\in F_{\ell+1}$ by the lemma below, since $(z,x_m)\in F_1$ by Remark \ref{r:start}.
\end{proof}

\begin{lem}\label{l:claim}Let $z=z_1\wedge z_2\wedge z_3$ for $z_j\in \Ltreort{x,v}$. If there is a path $v=v_0\to v_1\to\cdots\to v_\ell=x'$ in $\Lconn{i}(g)$, and $(z,v)\in F_m$, then $(z,x')\in F_{m+\ell}$.
\end{lem}

\begin{proof}
A straightforward induction shows $(z,v_j)\in F_{j+m}$ for $j=1,\ldots,\ell$.
\end{proof}

\subsection{Finishing the proof of Theorem \ref{t:vectfield2} for $i=2$}\label{ss:finishexact2}
In this section, $i=2$. Recall we write $S(V)=S(\pr_2(V))$, where $V\del H$.  To prove Lemma \ref{l:A=E}, we use the same strategy as for $i=1$, but the details are more complicated. First, we establish versions of Lemmas \ref{l:basis} and \ref{l:smh}:

\begin{rem}\label{r:basis}For $i=2$, we apply Lemma \ref{l:basis} to $H_2=\pr_2 H$ to get a basis $B_2(x')$ of $\pr_2(\ortQ{\pr_2(x),\pr_2(x')})$ with $(a)$ replaced by
\begin{equation}\label{e:a'}
	\set{a_2,b_2,\ldots,a_{k-1},b_{k-1}}\del B\del \F^1_{\max(k,2)}(g)^{(0)}\cup \set{b_2,\ldots, b_{k-1}}.
\end{equation}
A modification for $k=1$, needed to apply Cor. \ref{c:smh} below: If we assume $\gcdto(x,x',v_1,v_2,v_3,v_4)=1$, we can obtain $B_2(x')$ such that for $m=1,2,3,4$,
\begin{equation*}
	\gcdto(x,S(x',v_m,z_1,z_2,z_3))=1, \quad \text{for distinct }z_i\in B_2(x').
\end{equation*}
To do this, replace the proof of Lemma \ref{l:basis} by the following: Choose honest dual vectors $y$ to $\pr_2(x)$, $y'$ to $\pr_2(x')$ and $y_j$ to $\pr_2(v_j)$ by Prop. \ref{p:dual}. $S=\ideal{\pr_2(x),\pr_2(x')}$ is a summand, its dual is $D=\ideal{y,y'}$, and we have $T$ s.t. $(S\oplus D)\oplus T=H_2$. Using the dual vectors, we can modify $\pr_T(v_j)$ to $v_j'$, and $\pr_2(x)$, $\pr_2(x')$ to $p(x)$, $p(x')$, such that $p(x), p(x'), v_1'\ldots, v_4'$ is isotropic and $v_1'\ldots, v_4'$ extends to a symplectic basis by $y_1,\ldots,y_4$. Finally replace $p(x)$ by $p(x)-v_1'-v_2'-v_3'-v_4'$. The result is extendable to a symplectic basis $B_2(x')$ of $\ortQ{\pr_2(x),\pr_2(x')}$ that satisfies the desired equation.\qed
\end{rem}

Lemma \ref{l:smh} for $i=2$ is more complicated:

\begin{cor}\label{c:smh}
Let $g\ge 6$ and $n\le 3$. Let $B_2$ be a basis of $\pr_2(\ortQ{\pr_2(x)})$ satisfying \eqref{e:a'} in Remark \ref{r:basis}. Assume $z_1,\ldots,z_n\in B_2$. Then,
\begin{itemize}
	\item[$(i)$] For $k\ge 2$, $\Lconn{2}(g)$ is connected, and if $n\le 2$, then the complex with $a_1$ removed, i.e. $\Lconn{2}(g)\fra (\star(a_1)\fra \link(a_1))$, is also connected.
	\item[$(ii)$] For $k=1$, let $v_1,v_2\in \Lconnk{1}(g)$, and assume  $\gcd_2(x,v_1,v_2)>0$ and $\gcd_2(x,S(v_1,v_2,z_1,z_2,z_3))=1$. Then there exists a path $v_1\to w \to v_2$ in $\Lconnk{1}(g)$ with $\gcdto(x,v_1,v_2,w)>0$.
	\item[$(iii)$] For $k=0$, the full subcomplex of $\Lconnk{0}(g)$ spanned by the vertices $v$ with $\rk^{b_1}(x) =\rk^{b_1}(v)$ is connected. If $v_1,v_2\in \Lconnk{0}(g)$ with $\rk^{b_1}(x) =\rk^{b_1}(v_j)$ and $\gcdto(v_j,x)>0$, then there is a path $v_1\to w\to v_2$ in $\Lconnk{0}(g)$ with $\gcdto(x,v_j,w)>0$.
	\end{itemize}
\end{cor}
\begin{proof} For $k\ge 2$, the first part is easy since $a_1\in \Lconn{2}(g)$, and for any other $v\in \Lconn{2}(g)^{(0)}$, we see $\set{v,a_1}\in \Lconn{2}(g)$. For the second part of $k\ge 2$, let $v_1,v_2\in \Lconn{2}(g)$ be given. Note $\pr_2(x)=x-a_1$ is simple. By identifying $H_2\iso H(g)$, we can say $\pr_2(v_j)=v_j-a_1\in \Lconnxk{\pr_2 x}{k-1}(g)$, since $z_m\in H_2$ already. By Lemma \ref{l:smh}, $\Lconnxk{\pr_2 x}{k-1}(g)$ is connected when $g\ge 6$, so there is a path $\pr_2(v_1)=w_0'\to w_1'\to\cdots \to w_\ell'=\pr_2(v_2)$. Set $w_j=a_1+w_j'$, then $v_1=w_0\to w_1\to \cdots\to w_\ell=v_2$ is a path from $v_1$ to $v_2$ in $\Lconn{2}(g)$.

For $k=0$, let $v_1,v_2\in \Lconnk{0}(g)^{(0)}$ with $\rk^{b_1}(v_j)=\rk^{b_1}(x)$ be given. Since $v_j\in\F^2_1$, we get $\pr_2(v_j)\in \Lconnxk{x}{1}(g+1)$. Since $g\ge 6$, $\Lconnxk{x}{1}(g+1)$ is connected by Lemma \ref{l:smh}, and from the proof we get a path $\pr_2(v_1)\to w' \to\pr_2(v_2)$ in $\Lconnxk{x}{1}(g+1)$. Then $w'\in H_2$, and we set $w=a_1+\rk^{b_1}(x)b_1+w'\in \F^2_1$. By the choice of $w'$ in the proof, we see $v_1\to w\to v_2$ is a path in $\Lconnk{0}(g)$, and $\gcdto(x,w,v_j)=\gcdto(x,v_j)$, which also shows the last part of $k=0$.

Finally, for $k=1$, we proceed similarly to the proof of Lemma \ref{l:smh} for $k=1$, except that we consider $V^0=\ideal{x}$, $V^1=\ideal{v_1,v_2}$, and $V^2=\ideal{z_1,z_2,z_3}$. Projecting onto $H_2$, we take succesive duals similar to \eqref{e:duals}, which becomes
 \begin{equation}\label{e:duals2}
    (SV^0_{2}\oplus D^0)\oplus(SV^{1,0}_{2}\oplus D^1)\oplus (SV^{2,1}_{2}\oplus D^2)\oplus T^2=H_{2}.
\end{equation}
Since $\gcd_2(x,S(v_1,v_2,z_1,z_2,z_3))=1$, we have $SV^0_{2}=V^0_2=\ideal{\pr_2(x)}$ and $D^0=\ideal{u}$, where we can choose the dual $u\in H_2$ to be orthogonal to $v_1,v_2,z_1,z_2,z_3$. A dimension count shows that there is a simple vector $w'\in SV^{2,1}_{2}\oplus D^2$ such that $w'\perp z_j$ for $j=1,2,3$, and of course $w'\perp x,v_1,v_2$. Set $w=a_1+cu+w'$, where $c\in\Z$ is such that $w\perp x$. Then $w$ is orthogonal to $x,v_1,v_2,z_1,z_2,z_3$, and we see $v_1\to w\to v_2$ is a path in $\Lconnk{1}(g)$. Furthermore, by construction, $\gcd_2(x,v_1,v_2,w)=\gcd_2(x,v_1,v_2)>0$. 
\end{proof}

To prove Lemma \ref{l:A=E} for $i=2$, we now need two steps. The first asserts the existence of the desired basis of $\Ltreort{x,x',b_1}$, under certain conditions:

\begin{lem}\label{l:help} Let $x'\in\F^2_{k+1}$ be given. Assume there are $x_1',x_2',x_3',x_4'\in \F^2_{k+1}(g)^{(0)}$ such that $\set{x,x_j'}\in \F^2_k$ for $j=1,2,3,4$, and there is a basis $\streg B(x_j')$ for $\Ltreort{x,x_j',b_1}$ satisfying Lemma \ref{l:A=E} for $j=1,2,3,4$. Assume further:
\begin{itemize}
	\item[$(a)$]If $k\ge 2$, that $x_j'=A_j$ for $j\le k-1$ and $\gcdto(x,x',x_2', \ldots,x_4')>0$.
	\item[$(b)$]If $k=1$, that $\gcdto(x,x',x_1',\ldots,x_4')=1$. 
	\item[$(c)$]If $k=0$, that either: \quad $1)$ $\set{x,x',x_1',x_2',x_3',x_4'}\in\L^2(g)$, or \\$2)$ $\gcdto(x,x',x_1', \ldots,x_4')>0$ $($if $\pr_2(x)= 0$, omit $x$$)$, and furthermore $\rk^{b_1}(x_j')=\rk^{b_1}(x')=\rk^{b_1}(x)$ for $j=1,2,3,4$.
\end{itemize}
Then there is a basis $\streg B(x')$ of $\Ltreort{x,x',b_1}$ satisfying Lemma \ref{l:A=E}.
\end{lem}

\begin{proof}First, in case $(c\:1)$, the result follows directly from Lemma \ref{l:decompose}. So consider the other cases. We claim that if $\pr_2(x)\neq 0$, then
\begin{equation}\label{e:pr2decomp}
    \ortZ{x,x',b_1} = \pr_2(\ortZ{\pr_2(x),\pr_2(x')})\oplus \ideal{\tilde b_1}
\end{equation}
where $\tilde b_1=b_1-u$ for any $u\in H_2$ satisfying $	\ialg{x,u}=\ialg{x',u}=1.$
\eqref{e:pr2decomp} follows since $\rk^{a_1}(x)=\rk^{a_1}(x')=1$ and by a dimension count.

We take a $\Q$-basis $B_2(x')$ of $\pr_2(\ortZ{\pr_2 x, \pr_2 x'})$, as in Remark \ref{r:basis} with $v_j=x_j'$; note this is possible because of, respectively, $(a)$, $(b)$, or $(c\: 2)$. From it we get a $\Q$-basis $\overline B_2(x')$ for $\Ltre(\pr_2(\ortZ{\pr_2 x, \pr_2 x'}))$.

Write $B_2(x')=\set{z_1,\ldots,z_n}$. Now suppose for each pair $z_r, z_s\in B_2(x')$ with $r<s$ we have chosen $u_{rs}$ satisfying $\ialg{x,u_{rs}}=\ialg{x',u_{rs}}=1$. Set $\tilde b_1(r,s)=b_1-u_{rs}$. Then by \eqref{e:pr2decomp}, the following defines a basis of $\Ltreort{x,x',b_1}$:
 \begin{equation*}
    \streg B(x')=\streg B_2(x')\cup \set{\tilde b_1(r,s)\wedge z_r\wedge z_s\mid  1\le r<s \le n}
 \end{equation*}
We will show that for each $z\in \streg B(x')$ there is an $N$ such that $(x',z)\in F_N$, and at the same time construct the $u_{rs}$. 

First consider a given $z=z_r\wedge z_s\wedge z_t\in \streg B_2(x')$. Then there is $m\in\set{1,2,3,4}$ with $z\in \Ltreort{x,x_m',b_1}$. By assumption,  $(z,x_m')\in F_{q}$ for some $q\in\Z$. The assumptions $(a)$, $(b)$, or $(c\: 2)$ ensures that Cor. \ref{c:smh} gives a path $x_m'=w_0\to w_1\to\cdots\to w_\ell=x'$ in $\Lconn{2}(g)$ where $Z=\ortZ{z_r,z_s,z_t}$. (For $k=1$ use Remark \ref{r:basis}.) Then Lemma \ref{l:claim} gives $(z,x')\in F_{q+\ell}$. $(*)$

This finishes the proof for $\gcdto(x,x')=0$, since then $B(x')=B_2(x')$.

Now consider the other type of basis vector in $\streg B(x')$, so let $r<s$ be given. Take $m\in\set{1,2,3,4}$ with $z_r\wedge z_s\in\Lort{2}{x,x_m',b_1}$.

For now, assume $k\neq 1$. By Cor. \ref{c:smh} with $Z=\ortZ{z_r,z_s}$, there is a path $x_m'=v_0\to v_1\to \cdots\to v_\ell=x'$ in $\Lconn{2}(g)$ with $\gcd_2(x,v_j,v_{j-1})>0$; for $k\ge 2$ this is because the path avoids $a_1$. Thus by Prop. \ref{p:dual}, if we take $\Q$-coefficients, we can choose $u_{j}\in (H_\Q)_2$ such that 
\begin{equation}\label{e:uproperties}
	\ialg{x,u_{j}}=\ialg{v_j,u_j}=\ialg{v_{j-1},u_j}=1.
\end{equation}
Now $\tilde b_1^j:= b_1-u_j\in\ortZ{x,v_j,v_{j-1}}$. For $j=0$ just use $v_0=x_m'$ to get $\tilde b_1^0\in\ortZ{x,x_m'}$. We set $u_{r,s}=u_\ell$.

Write $z^j=\tilde b_1^j\wedge z_r\wedge z_s$. Inductively assume there is $n(j)\in\N$ such that $(z^j,w_j)\in F_{n(j)}$. The induction start is the assumption on $\streg B(x'_m)$. For the induction step, suppose $(z^j,v_j)\in F_{n(j)}$. Consider the difference $z^{j+1}-z^j=(u_{j+1}-u_j)\wedge z_2\wedge z_3$. Since $(u_{j+1}-u_j),z_r,z_s\in \pr_2(\ortQ{\pr_2(x),\pr_2(v_j)})$, we can apply $(*)$ above for $x'=v_j$ and get $p(j)\in\N$ with $(z^{j+1}-z^j,v_j)\in F_{p(j)}$. Combined with the induction hypothesis, we see $(z^{j+1},v_j)\in F_{\max(p(j),n(j))}$. Since $\set{x,v_j,v_{j+1}}$ is a simplex, $(z^{j+1},v_{j+1})\in F_{n(j+1)}$ where $n(j+1)=\max(p(j),n(j))+1$. This finishes the induction. We have shown that $z=\tilde b_1(r,s)\wedge z_r\wedge z_s$ satisfies $(z,x')\in F_{n(\ell)}$, 

For $k=1$, the path $x_m'\to w\to x'$ provided by Cor. \ref{c:smh} satisfies $\gcdto(x,x_m',w,x')>0$, so there is now $u\in (H_\Q)_2$ with
	\[\ialg{x,u}=\ialg{x_m',u}=\ialg{w,u}=\ialg{x',u}=1.
\]
Then choose $u_{r,s}=u$. Write $z=(b_1-u)\wedge z_r\wedge z_s$. Lemma \ref{l:claim} gives that $(z,x')\in F_{N}$ for some $N$, because $\streg B(x_m')$ satisfies Lemma \ref{l:A=E}.
\end{proof}

With this, we can show Lemma \ref{l:A=E} for $i=2$:

\begin{proof}[Proof of Lemma \ref{l:A=E} for $i=2$] First, if $k\ge 2:$  We have $x_1,\ldots,x_5\in \F^2_{k+1}$ from Lemma \ref{l:1234}, and we use as $x_1',\ldots,x_4'$ four among $x_1,\ldots,x_5$ such that Lemma \ref{l:help}$(a)$ holds. From Remark \ref{r:start} there are bases $\streg B(x_m)$ as in Lemma \ref{l:A=E}. Now Lemma \ref{l:help} gives the basis $\streg B(x')$ as desired. 

If $k=1$ we consider two cases: First, if $\gcdto(x,x',x_1,x_2,x_3,x_4)=1$, then we argue as in the case $k\ge 2$ with $x_j'=x_j$ for $j=1,2,3,4$, and are done.

If not, then we will find $x_1',\ldots,x_4'\in\F^2_{2}$ such that:
\begin{equation}\label{e:k1}
\gcdto(x,x_1',x_2',x_3',x_4',x_1,x_2,x_3,x_4)=\gcdto(x,x',x_1',x_2',x_3',x_4')=1
\end{equation}
Then Lemma \ref{l:help} gives the basis $\streg B(x'_j)$ for each $j=1,2,3,4$ by using $x_1,x_2,x_3,x_4$, which by Lemma \ref{l:help} along with \eqref{e:k1} gives the basis $\streg B(x')$.

To construct these $x_j'$, let $S=S(x,x',x_1,\ldots,x_4,y)\del H_2$, where $y\in H_2$ denotes a dual vector to $\pr_2(x)$, which is simple since $k=1$.  We can employ $\pr_2(x),\pr_2(x_1),\ldots,\pr_2(x_4)$ as part of a $\Z$-basis of $S$, since 
$\set{x,x_1,x_2,x_3,x_4}\in\F^2_{1}$. Obviously, $\dim S\le 7$, and since $\dim H_2=2g\ge 12$, we can extend this basis of $S$ by four basis vectors, call them $v_1,v_2,v_3,v_4$. Set $x_j'=a_1+v_j+c_jy$ for $j=1,2,3,4$. Here $c_j\in \Z$ is chosen such that $x_j'\perp x$, i.e. $c_j=-\ialg{x,a_1+v_j}$. Now \eqref{e:k1} follows from the fact that 
\begin{equation*}
	1=\gcdto(S,v_1,v_2,v_3,v_4)=\gcdto(S,x_1',x_2',x_3',x_4').
\end{equation*}

Last, if $k=0$, we again consider two cases. First, if $\rk^{b_1}(x')=\rk^{b_1}(x)$. From Lemma \ref{l:1234} $(iv)$ we have $\rk^{b_1}(x_j)=\rk^{b_1}(x)$ for $j=1,2,3,4,5$, so we can do as in $k\ge 2$.

If not, we get from Lemma \ref{l:1234} that $\set{x,x'}^*=\set{x_1',\ldots,x_4'}$ satisfies, 
\begin{equation}\label{}
	\set{x,x',x_1',\ldots,x_4'}\in \F^2_0, \text{ and } 	\rk^{b_1}(x_j')=\rk^{b_1}(x).
\end{equation}
Then by the first part of $k=0$, we have a basis $\streg B(x_j')$ for $j=1,2,3,4$, and thus by Lemma \ref{l:help} we also get $\streg B(x')$ satisfying Lemma \ref{l:A=E}.
\end{proof}

\end{document}